\documentclass[a4paper, 11pt]{article}
\usepackage{amsmath}
\usepackage{amssymb,esint}
\usepackage{amscd}
\usepackage{mathrsfs}
\usepackage{xspace}
\usepackage{fancyhdr}
\usepackage{xcolor}
\usepackage{verbatim}
\usepackage{cite}
\usepackage{appendix}
\usepackage{srcltx} 

\newtheorem{theorem}{Theorem}[section]

\newtheorem{corollary}[theorem]{Corollary}

\newtheorem{lemma}[theorem]{Lemma}

\newtheorem{proposition}[theorem]{Proposition}
\newtheorem{remark}[theorem]{Remark}

\newenvironment{proof}[1][Proof]{\textbf{#1.} }{\hfill\rule{0.5em}{0.5em}}
{\catcode`\@=11\global\let\AddToReset=\@addtoreset
	\AddToReset{equation}{section}
	
	\AddToReset{theorem}{section}

\begin{document}
		\title{Nonlinear Landau damping for the 2d Vlasov-Poisson system with massless electrons around Penrose-stable equilibria}
		\author{
			{\bf Lingjia Huang\thanks{E-mail address: ljhuang20@fudan.edu.cn, Fudan University, 220 Handan Road, Yangpu, Shanghai, 200433, China.}, Quoc-Hung Nguyen\thanks{E-mail address: qhnguyen@amss.ac.cn, Academy of Mathematics and Systems Science,
					Chinese Academy of Sciences,
					Beijing 100190, PR China.} ~and~Yiran Xu\thanks{E-mail address: yrxu20@fudan.edu.cn, Fudan University, 220 Handan Road, Yangpu, Shanghai, 200433, China. }}}
		\date{}  
		\maketitle
		\begin{abstract}
			In this paper, we prove the nonlinear asymptotic stability of the Penrose-stable equilibria among solutions of the $2d$ Vlasov-Poisson system with massless electrons.
		\end{abstract}
		\section{Introduction}
	This paper is devoted to study the Vlasov-Poisson system of the form:
		\begin{equation}\label{toymodel}
			\begin{cases}
				\partial_t \mathbf{f} + v\cdot \nabla_x  \mathbf{f}+E\cdot\nabla_v  \mathbf{f}=0,\\
				E=-\nabla_xU,    -\Delta U+U =\varrho-1+A(U),~~ \varrho(t,x)=\int_{\mathbb{R}^2} \mathbf{f}(t,x,v)dv,\\
			\end{cases}
		\end{equation}
		on the whole space $x\in\mathbb{R}^2$, $v\in\mathbb{R}^2$, where $\mathbf{f}=\mathbf{f}(t,x,v)\ge 0$ is the probability distribution of charged particles in plasma, $\varrho(t,x)$ is the electric charge density, and  $E=E(t,x)$ is electric field and $A:\mathbb{R}\to \mathbb{R}$  is smooth and  satisfies  $A(r)=\mathcal{O}(r^2)$ as $r\to 0$. In particular, the system is for massless electrons or ions when $A(r)= r+1-e^r$.
		This system was extensively studied (\cite{arsenev,batt,EHorst1982,CBardos1985,FBouchut1991,DHanKwan2011,CBardos2018,KPfaffelmoser1992,JSchaeffer1991,EHorst1993,RTGlassey1996,LionsPL1991,HJHwang2011,SHCHoi2011,JSmulevici2016,Kwanlacobelli2017,Xwang2018,GriffinLacobelli,GuoPausader,Alonescwau2020,FlyOuPau}), focusing on the global existence, regularity results and  longtime behavior of solutions. We are interested in the asymptotic stability  of solutions $f_i$ to \eqref{toymodel}  in the form
$$\mathbf{f}(t,x,v)=\mu(v)+f(t,x,v)$$
		where $\mu(v)$ is a stable equilibrium with $\int_{\mathbb{R}^2}\mu(v)dv=1$ and $\mathbf{f}(t=0,x,v)$ closes to $\mu(v)$.  So, $f$ solves the following perturbed system
		\begin{equation}\label{eq2}
			\begin{cases}
				\partial_t f + v\cdot \nabla_x f +E\cdot\nabla_v \mu=-E\cdot\nabla_v f,\\
				E=-\nabla_xU,    -\Delta U+U =\rho+A(U),~~ \rho(t,x)=\int_{\mathbb{R}^2}f(t,x,v)dv,\\
				f|_{t=0}=f_0.
			\end{cases}
		\end{equation}
	The following are assumptions on $\mu$ and $A$ in this paper. 
	\begin{itemize}
		\item \textit{Assumption 1:}	$\mu$ satisfies the \textit{Penrose stability condition}:
		\begin{equation}\label{penrose}
			\inf_{\tau\in\mathbb{R},\xi\in\mathbb{R}^2}\left|1-\int_{0}^{+\infty}e^{-i\tau s}\frac{1}{1+|\xi|^2}i\xi\cdot\widehat{\nabla_v\mu}(s\xi)ds\right|\geq\text{\r{c}},
		\end{equation}
		for some constant $\text{\r{c}}>0$, $\widehat{\nabla_v\mu}$ is the Fourier transform of $\nabla_v\mu$ in $\mathbb{R}^2$.
		\item \textit{Assumption 2:} $\mu\in L^1$ satisfies 
		$$\|\langle \cdot\rangle^3\nabla_v\mu(\cdot)\|_{W^{2,\infty}}+\|\langle \cdot\rangle^{7}\nabla\mu(\cdot)\|_{W^{11,1}}\leq M^*,$$
		for some $M^*<\infty$.
		\item \textit{Assumption 3:}  $A:\mathbb{R}\to \mathbb{R}$  is $C^3$ and  satisfies  
		\begin{align*}
			\sup_{|r|\leq 1}	\left(\left|\frac{A(r)}{r^2}\right|+\left|\frac{A'(r)}{r}\right|+\left|{A''(r)}\right|+\left|{A^{(3)}(r)}\right|\right)\leq C_A,
		\end{align*}
		for some constant $C_A>0$.
	\end{itemize}
		Note that a particular example for the Assumption 3 is $A(r)=r+1-e^r$ corresponding to the Vlasov–Poisson system with electrons mass, see  \cite{FBouchut1991}.\\
	Under the Penrose condition \eqref{penrose}, Landau damping was studied in the breakthrough paper  \cite{CMouhot2011} by Mouhot and Villani in the case $\mathbb{T}^d\times\mathbb{R}^d$ with Gevrey or analytic data (see also \cite{JBedrossian2016,NguyenTT2020} for refinements and simplifications). In \cite{Bedrossiantunis} Bedrossian showed that the results therein do not hold in finite regularity (see also \cite{NguyenTT2020b})). We note that related mechanisms in the fluid are the vorticity mixing by shear flows   \cite{JbedIhes2015,ALonescu2020,Alonescu2020acta,Naderweiren}.
	In the whole space  $\mathbb{R}^d\times\mathbb{R}^d$ with $d\geq 3$, it was established for the screened Vlasov-Poisson system (i.e., $A(U)=0$) in \cite{JBedrossian2018} by Bedrossian, Masmoudi, and Mouhot with  Sobolev data. Their proof relies on the dispersive mechanism in Fourier space to control the plasma echo resonance. However, a decay in time of their result is far from optimal.  In \cite{HanKwanD2021}, Han-Kwan, T. Nguyen, and Rousset used pointwise dispersive estimates of the linearized system of \eqref{eq2} to obtain the decay estimates for the density $\rho$ as follows:
		\begin{equation}\label{x1}
		\sum_{j=0,1}\left[(1+t)^{j}\|\nabla^i \rho(t)\|_{L^1}+(1+t)^{d+j}\|\nabla^i\rho(t)\|_{L^\infty}\right]\lesssim\varepsilon_0 \log(t+2),~~\forall~t>0,
		\end{equation}
with $d\geq 3$. Note that \eqref{x1} is optimal up to a logarithmic correction. Also, higher derivatives for the density were established in \cite{NguyenTTr2020}. Note that the problem \eqref{eq2} in dimension $d=2$ is \textit{critical and open}. Recently, in \cite{AIonescu2022} Ionescu, Pausader, Wang, and Widmayer proved the first asymptotic stability result for the unscreened Vlasov-Poisson system (i.e., $A(U)=U$) in $\mathbb{R}^3$ around the Poisson equilibrium, see also in \cite{Pausader2021} for the case of a repulsive point charge. The unscreened case is open for the general equilibria. However, in \cite{HNRcmp,JBedrossianNa2020}  they studied the linearized unscreened Vlasov-Poisson equation around suitably stable analytic homogeneous equilibria in $\mathbb{R}^d_x \times \mathbb{R}^d_v$.\\
Very recently, authors in \cite{HNX1} established new estimates and cancellations of the kernel to the linearized problem ( see Proposition \ref{ProG}) and proved the sharp decay estimates for the density $\rho$ in Besov spaces:
\begin{align*}
	\sum_{j=0,1}\sum_{p=1,\infty}\left((1+t)^{j+\frac{d(p-1)}{p}}\|  \nabla^j\rho(t)\|_{L^p}+(1+t)^{j+a+\frac{d(p-1)}{p}}\|  \nabla^j\rho(t)\|_{\dot{B}^a_{p,\infty}}\right)\lesssim \varepsilon_0,~~\forall t>0,
\end{align*} with $d\geq 3$, 
for some $a\in (0,1)$. In particular, this implies \eqref{x1} without logarithmic correction. \vspace{0.2cm}\\
Our goal in this paper is to extend our work in  \cite{HNX1} to  dimension $d=2$.\vspace{0.2cm}\\
Before we state our theorem, we need to recall some notations in    \cite{HNX1}. 
For $\gamma\in (0,1)$ and $m\in \mathbb{N}$, and $g:[0,\infty)\times\mathbb{R}^2\to \mathbb{R}$, we define for $T>0$,
		$$
		\|g\|_{ m+\gamma,T}=\sum_{j=0}^m\sum_{p=1,\infty}\sup_{s\in [0,T]}\left(\langle s\rangle^{\frac{2(p-1)}{p}}\|  g(s)\|_{L^p}+\langle s\rangle^{j+\gamma+\frac{2(p-1)}{p}}\|  \nabla^jg(s)\|_{\dot{B}^\gamma_{p,\infty}}\right),$$
		with 
			$
		\|g\|_{\dot{B}_{p,\infty}^{s}}:=\sup_{\alpha}\frac{\|\delta_\alpha g(x)\|_{L^{p}_x}}{|\alpha|^s}.$ Moreover, $	\|g\|_{ m+\gamma}:=\|g\|_{ m+\gamma,T}$ when $T=\infty$.\vspace{0.15cm}\\
			For  $h_1:\mathbb{R}^2_x\times \mathbb{R}^2_v\to  \mathbb{R}^m$, we define for $(x,v)\in\mathbb{R}^2_x\times \mathbb{R}^2_v$ and $a\in (0,1)$,
			\begin{align*}
				&	\dot{\mathcal{D}}^a_1h_1(x,v)=\sup_{z\in\mathbb{R}^2}\frac{|h_1(x,v)-h_1(x-z,v)|}{|z|^a}, ~~	\dot{\mathcal{D}}^a_2h_1(x,v)=\sup_{z\in\mathbb{R}^2}\frac{|h_1(x,v)-h_1(x,v-z)|}{|z|^a},\\
				&
				\mathcal{D}^ah_1=|h_1|+\dot{\mathcal{D}}^a_1h_1+\dot{\mathcal{D}}^a_2h_1.
			\end{align*}
			For $h_2:\mathbb{R}^2_x\times \mathbb{R}^2_v\to  \mathbb{R}$ and $a\in (0,1)$ we define 
			\begin{equation}\label{zzz16}
				||| h_2|||_{1+a}:=	\sum_{p=1,\infty}\sum_{i=0,1}\left\|\mathcal{D}^{a} (\nabla_{x,v}^ih_2)\right\|_{L^1_{x}L^p_v\cap L^1_{v}L^p_x}.
			\end{equation}
		Our main result is as follows.
		\begin{theorem}\label{Thmmain}
			Let $a\in (0,1)$. There exist $C_0>0$,  $\varepsilon\in (0,1)$ such that if $
			|||f_0|||_{1+a}\leq\varepsilon,$
			and 
		$
				\liminf_{\kappa\to 0} 	|||f_0-f_0^\kappa|||_{1+a}=0$
		for some sequence $f_0^\kappa\in C^\infty_c(\mathbb{R}^2_x\times \mathbb{R}_v^2)$. Then the problem \eqref{eq2} has a unique global solution $f$ with 
			\begin{equation*}
				\|\rho\|_{1+a}+	\|U\|_{1+a}\leq C_0 	|||f_0|||_{1+a}.
			\end{equation*}
		\end{theorem}
			\begin{remark}Our method does not work in the $1d$ dimension case because the density $\rho$ is not decaying enough to estimate for characteristics.  
		\end{remark}
		\begin{remark}
			While finishing our work, we learned that Toan Nguyen was working on this problem and had a similar result to ours. But the two works are independent.
		\end{remark}
	As \cite[Corollary 1.1]{HanKwanD2021}, thanks to Theorem \ref{Thmmain}, we also obtain the following scattering property for the solution to \eqref{eq2}. The proof is omitted as it is analogous to \cite[Proof of Corollary 1.1]{HanKwanD2021}.
		\begin{corollary}\label{cor1} With the same assumptions and notations as in Theorem \ref{Thmmain}, there is  a  function $f_\infty\in W^{1,\infty}$ given by 
			$$f_\infty(x,v)=f_0(x+Y_\infty(x,v),v+W_\infty(x,v))+\mu(v+W_\infty(x,v))-\mu(v),$$ such that 
			$$\sup_{t>0}\langle t\rangle\|f(t,x+tv,v)-f_\infty(x,v)\|_{L^\infty_{x,v}}\lesssim 	|||f_0|||_{1+a},~~\|Y_\infty\|_{L^\infty_{x,v}}+\|W_\infty\|_{L^\infty_{x,v}}\lesssim	|||f_0|||_{1+a}.$$
		\end{corollary}
		The following is a key proposition of this paper.  Let $T>0$,  $\mathfrak{g}:[0,T]\times \mathbb{R}^2\to \mathbb{R}$.  One considers
		\begin{equation}
			\mathbf{E}(t,x)=-\nabla (-\Delta +1)^{-1}(\mathfrak{g})(t,x), ~~||\mathfrak{g}||_{1+a,T}<\infty.
		\end{equation}
		Let $\left(X_{s,t}(x,v),V_{s,t}(x,v)\right)$  be the flow associated to the vector field $(v,\mathbf{E}(t,x))$, i.e.
		\begin{equation}\label{X and V}
			\begin{cases}
				\frac{d}{ds}X_{s,t}(x,v)=V_{s,t}(x,v),\qquad&X_{t,t}(x,v)=x,\\
				\frac{d}{ds}V_{s,t}(x,v)=	\mathbf{E}(s,X_{s,t}(x,v)),\qquad &V_{t,t}(x,v)=v,
			\end{cases}
		\end{equation} 
		for any $0\leq s\leq t\leq T$, $(x,v)\in \mathbb{R}^2\times\mathbb{R}^2$. Hence, 
		\begin{equation}\label{z4}
			X_{s,t}(x,v)=x-(t-s)v+Y_{s,t}(x-vt,v),~~V_{s,t}(x,v)=v+W_{s,t}(x-vt,v),
		\end{equation}
		where 
		\begin{align}\label{definition Y W}
			\begin{split}
				&Y_{s,t}(x,v)=\int_{s}^{t}(\tau-s)	\mathbf{E}(\tau,x+\tau v+Y_{\tau,t}(x,v))d\tau,\\&
				W_{s,t}(x,v)=-\int_{s}^{t}	\mathbf{E}(\tau,x+\tau v+Y_{\tau,t}(x,v))d\tau.
			\end{split}
		\end{align}
		\begin{proposition}\label{estimates about Y and Wxx} Let $a\in (0,1)$, then there exists $\varepsilon_0\in(0,1)$ such that  for any $\|\mathfrak{g}\|_{1+a,T}\leq \varepsilon_0$, we have
			\begin{align}	\sup_{0\leq s\leq t\leq T}\sup_{\alpha}	\frac{	\|\delta_{\alpha}^v	\nabla_vY_{s,t}\|_{L^\infty_{x,v}}}{|\alpha|^{a}}+\sup_{0\leq s\leq t\leq T}\langle s\rangle\sup_{\alpha}	\frac{	\|\delta_{\alpha}^v	\nabla_vW_{s,t}\|_{L^\infty_{x,v}}}{|\alpha|^a}\lesssim_a \|\mathfrak{g}\|_{1+a,T},\label{highestresults of Y smallxx}
		\end{align}
where 
			\begin{align*}
				\delta_{\alpha}^vY_{s,t}(x,v)=Y_{s,t}(x,v)-Y_{s,t}(x,v-\alpha),~~ \delta_{\alpha}^vW_{s,t}(x,v)=W_{s,t}(x,v)-W_{s,t}(x,v-\alpha).
			\end{align*}
		\end{proposition}
	\begin{remark} Estimate of the first term (in LHS) in \eqref{highestresults of Y smallxx} is critical. It is a key estimate in the  proof of Theorem \ref{Thmmain}. 
	\end{remark}
The strategy to prove Theorem \ref{Thmmain} in this paper is slightly different from that in our previous paper \cite{HNX1}. First, we use the fixed-point theorem with a local in-time norm to prove the local existence result (see Proposition \ref{propolocal}), then we establish a suitable bootstrap property for the problem \eqref{eq2} (see Proposition  \ref{propoglobal}). We finish the proof of Theorem \ref{Thmmain} by using a bootstrap argument.\vspace{0.2cm}\\
The paper is organized as follows. In section 2, we establish the equivalence of $\rho$ and $f$, then we state our bootstrap property and local existence result; in the last part of this section, we prove Theorem \ref{Thmmain}. In section 3, we prove pointwise estimates of characteristics $(Y_{s,t},W_{s,t})$ in \eqref{definition Y W}.  In section 4, we estimate the contribution of the initial data and the reaction term. 
In Section 5, we prove our bootstrap property and local existence result. 
		\vspace{0.4cm}\\
		\textbf{Acknowledgements:} 
		Q.H.N.  is supported by the Academy of Mathematics and Systems Science, Chinese Academy of Sciences startup fund, and the National Natural Science Foundation of China (No. 12050410257 and No. 12288201) and  the National Key R$\&$D Program of China under grant 2021YFA1000800. Q.H.N. also wants to thank Benoit Pausader for his stimulating comments and suggestion to consider the Vlasov-Poisson system with massless electrons.
		\section{Dynamics of the density and bootstrap argument}
		\label{secboot}
		In this section, we recast  the system  \eqref{eq2} as a problem for the density $\rho$. Let  $\left(X_{s,t}(x,v),V_{s,t}(x,v)\right)$  be the flow associated to the vector field $(v,E(t,x))$. 
		Then, the solution $f$ of the equation \eqref{eq2} is given by 
		\begin{equation}\label{t1}
			f(t,x,v)=f_0(X_{0,t}(x,v),V_{0,t}(x,v))-\int_{0}^{t}E(s,X_{s,t}(x,v))\cdot\nabla_v\mu(V_{s,t}(x,v))ds,
		\end{equation}
		(see (1.3) in \cite{Hung2021} with $\mathbf{B}(t,x,v)=(v,E(t,x))$).  Set $\mathfrak{g}=\rho+A(U)$, so $E=-\nabla_x(-\Delta+1)^{-1}\mathfrak{g}.$
		We denote
		\begin{align}\label{I R}
			\begin{split}
				&\mathcal{I}_{f_0}(\mathfrak{g})(t,x)=	\int_{\mathbb{R}^2}f_0(X_{0,t}(x,v),V_{0,t}(x,v))dv,\\
				&\mathcal{R}(\mathfrak{g})(t,x)=\int_{0}^{t}\int_{\mathbb{R}^2}\left(E(s,x-(t-s)v)\cdot\nabla_v\mu(v)-E(s,X_{s,t}(x,v))\cdot\nabla_v\mu(V_{s,t}(x,v))\right)dvds.
			\end{split}
		\end{align}
	As \cite[(2.5)]{HNX1}, we obtain that $\rho(t,x)=\int_{\mathbb{R}^2}f(t,x,v)dv$ is the solution of the following equation:
		\begin{equation}\label{equation rho}
			\rho=G*_{(t,x)}(\mathcal{I}(\mathfrak{g})+\mathcal{R}(\mathfrak{g})+A(U))+\mathcal{I}(\mathfrak{g})+\mathcal{R}(\mathfrak{g}),
		\end{equation}
		where  $G$ is the kernel satisfying:
		\begin{equation}\label{definition of G about K}
			\widetilde{G}=\frac{\widetilde{K}}{1-\widetilde{K}},~~
			\widetilde{K}(\tau,\xi)=\int_{0}^{+\infty}e^{-i\tau t}\frac{1}{1+|\xi|^2} i\xi\cdot\widehat{\nabla_v\mu}(t\xi)dt,
		\end{equation}
	and $\widetilde{\cdot}$ is the "space-time" Fourier transform on $\mathbb{R}^{2}\times\mathbb{R}$.
		Thanks to the Penrose condition, one has 
		$
		|1-\widetilde{K}|\geq \text{\r c}>0,
		$ which implies that 
		$\widetilde{G}$ is well defined.\vspace{0.1cm}\\
		Define for $T>0$
		$$\mathcal{V}_{\varepsilon}:=\{h\in L^1\cap L^\infty(\mathbb{R}^2):||h||_{L^1\cap L^\infty(\mathbb{R}^2)}\leq\varepsilon \}.$$
			By the standard argument, one has 
		\begin{lemma}\label{lem0} Let $a\in (0,1)$. There exist $C_0\geq 1$ and $\tilde{\varepsilon}_0\in (0,1)$ such that if $\varrho\in \mathcal{V}_{\tilde{\varepsilon}_0}$, the problem 
			\begin{equation}
				-\Delta u+u=\varrho+A(u)
			\end{equation}
			has a unique soluion $u$ in $\mathcal{V}_{c\tilde{\varepsilon}_0}$ satisfying 
			\begin{equation}\label{zzz18}
			\sum_{j=0}^{2}	\sum_{p=1,\infty}\left(	||(u,A(u))||_{L^p}+||\nabla^j(u,A(u))||_{\dot{B}^{\frac{a}{2}}_{p,\infty}}\right)\leq C_0\sum_{p=1,\infty}\left(	||\varrho||_{L^p}+||\varrho||_{\dot{B}^{a}_{p,\infty}}\right).
			\end{equation}
			In addition, we can define a map $\mathcal{N}:\mathcal{V}_{\tilde{\varepsilon}_0}\to \mathcal{V}_{c\tilde{\varepsilon}_0}$: $u=\mathcal{N}(\varrho)$ for any $\rho\in \mathcal{V}_{\tilde{\varepsilon}_0}$. The map $\mathcal{N}$ satisfies 
			\begin{align*}
				&\sum_{j=0}^{2}	\sum_{p=1,\infty}\left(	||\mathcal{N}(\varrho_1)-\mathcal{N}(\varrho_2)||_{L^p}+||\nabla^j(\mathcal{N}(\varrho_1)-\mathcal{N}(\varrho_2))||_{\dot{B}^{\frac{a}{2}}_{p,\infty}}\right)\\&\quad\quad\quad\quad\quad\quad\quad\quad\quad\quad\quad\quad\quad\quad\quad\quad\quad\quad\leq C_0\sum_{p=1,\infty}\left(	||\varrho_1-\varrho_2||_{L^p}+||\varrho_1-\varrho_2||_{\dot{B}^{a}_{p,\infty}}\right).
			\end{align*}
		\end{lemma}
Thanks to Lemma \ref{lem0}, $\mathfrak{g}$ and $U$ in \eqref{equation rho} can be performed via $\rho$: 
	\begin{equation}\label{zzz15}
		\mathfrak{g}(t)=\rho(t)+A(\mathcal{N}(\rho)(t)), U(t)=\mathcal{N}(\rho(t)),
	\end{equation}
provided that $\rho(t)\in \mathcal{V}_{\tilde{\varepsilon}_0}$. 
In particular, we can write \eqref{equation rho} as an equation for the density $\rho$ in $[0,T]$ when $\rho(t)\in \mathcal{V}_{\tilde{\varepsilon}_0}$ for any $t\in [0,T]$. \vspace{0.15cm}\\
		The following is  the boundedness of the operator $\mathcal{G}:=G*_{(t,x)}$ in \cite[ Theorem 3.9]{HNX1}. 	\begin{proposition} \label{ProG}There holds 
			\begin{equation}
			\|G*_{(t,x)}g\|_{1+a, T}\leq C	||g\|_{1+a,T},
			\end{equation}
		where $C=C(\|\langle \cdot\rangle^{7}\nabla\mu(\cdot)\|_{W^{11,1}},a,\text{\r c} ).$
		\end{proposition}
		The following is our main bootstrap proposition.
		\begin{proposition} \label{propoglobal} 	Let $a\in (0,1)$. There exist $C_1\geq 1$,   $\varepsilon_1\in (0,1)$ such that  if 
		$|||f_0|||_{1+a}<\infty$ and the problem \eqref{eq2} has a unique solution $(f,U)$ in $[0,T]$ for some $T<\infty$ with 
			\begin{equation}\label{z1}
				\|\rho\|_{1+a,T}+\|U\|_{1+a,T}\leq \varepsilon_1.
			\end{equation}
			Then, we have 
			\begin{equation*}
				\|\rho\|_{1+a,T}+	\|U\|_{1+a,T}\leq C_1|||f_0|||_{1+a}.
			\end{equation*}
		\end{proposition}
	Here is the local well-posedness result. 
\begin{proposition} \label{propolocal}
	Let $a\in (0,1)$. There exists $\varepsilon_2>0$ such that if $\rho(0,x)=\int_{\mathbb{R}^2}f_0(x,v)dv$ satisfies 
	\begin{align}
		\label{zzz5}
		&\sum_{p=1,\infty} ||\rho(0)||_{L^p}+ ||
		\nabla \rho(0)||_{\dot B^{a}_{p,\infty}}\leq \varepsilon_2,\\&
		\label{f_0condition}
		||| f_0|||_{1+a}<\infty, \quad
		\lim_{\kappa\to 0}||| f_0-f_0^\kappa|||_{1+a}=0,
	\end{align}
	for some sequence $f_0^\kappa\in C^\infty_c(\mathbb{R}^2_x\times \mathbb{R}_v^2)$. Then, the problem \eqref{eq2} has a unique local solution $f$ in $[0,T] $ with 
	\begin{align}
		\label{zzz8}
	&	\|\rho\|_{1+a,T}+ \|U\|_{1+a,T}\leq C_2\left(\sum_{p=1,\infty} ||\rho(0)||_{L^p}+ ||
	\nabla \rho(0)||_{\dot B^{a}_{p,\infty}}\right),\\&
	\label{zzz77}
	\lim_{t\to 0^+} \left(\|\rho-\rho(0)\|_{1+a,t}+ \|U-U(0)\|_{1+a,t}\right)=0,
	\end{align}
	for some $T\in (0,1)$ where $C_2\geq 1$  only depends on $a$ and $C_A$.
\end{proposition}
With  Proposition \ref{propoglobal}  and Proposition \ref{propolocal} in hand,  we can obtain Theorem \ref{Thmmain}. \\
		\begin{proof}[Proof of Theorem \ref{Thmmain}] By \eqref{zzz16}, one has 
				\begin{equation}\label{zzz17}
					\sum_{p=1,\infty}	||\rho(0)||_{L^p}+	||
					\nabla \rho(0)||_{\dot B^{a}_{p,\infty}}\leq C_3 |||f_0|||_{1+a},
				\end{equation}
			for some $C_3\geq 1$. Let $C_1\geq 1$,   $\varepsilon_1\in (0,1)$ be in  Proposition \ref{propoglobal} and $C_2\geq 1$,   $\varepsilon_2\in (0,1)$ be in  Proposition  \ref{propolocal} 
$$T^*=\sup\left\{T: \text{the solution $f$ of \eqref{eq2}  exists on} ~~[0,T], ~~\text{s.t}~	\|\rho\|_{1+a,T}+	\|U\|_{1+a,T}\leq M|||f_0|||_{1+a}\right\},$$
		with $M=100C_1C_2C_3$.
	Assume
		\begin{equation}\label{zzz6}
			|||f_0|||_{1+a}\leq \frac{\min\{\varepsilon_1,\varepsilon_2\}}{100C_1C_2C_3}.
		\end{equation}
So, thanks to Proposition \ref{propolocal}, the problem \eqref{eq2} has a unique local solution $f$ in $[0,T] $ satisfying
\begin{equation*}
		\|\rho\|_{1+a,T}+ \|U\|_{1+a,T}\leq C_2\left(\sum_{p=1,\infty} ||\rho(0)||_{L^p}+ ||
	\nabla \rho(0)||_{\dot B^{a}_{p,\infty}}\right)\overset{\eqref{zzz17}}\leq M|||f_0|||_{1+a}.
\end{equation*}
This gives $T^*>0.$ Now we suppose that $T^*<\infty$. By \eqref{zzz6} one has $	\|\rho\|_{1+a,T^\star}+	\|U\|_{1+a,T^\star}\leq \varepsilon_1.$ So, we can apply Proposition  \ref{propoglobal} to  get that 
\begin{equation}
	\label{zzz10}\|\rho\|_{1+a,T^\star}+	\|U\|_{1+a,T^\star}\leq C_1|||f_0|||_{1+a}.
\end{equation}
	In particular, 
$$
		\sum_{p=1,\infty}	||\rho(T^\star)||_{L^p}+	||
		\nabla \rho(T^\star)||_{\dot B^{a}_{p,\infty}}\leq C_1|||f_0|||_{1+a}\leq  \varepsilon_2.$$
Moreover, by Remark  \ref{Re} one has $|||f(T^\star)|||_{1+a}<\infty $ and 	$\lim_{\kappa\to 0}||| f(T^\star)-f_{T^\star}^\kappa|||_{1+a}=0$
for some sequence $f_{T^\star}^\kappa\in C^\infty_c(\mathbb{R}^2_x\times \mathbb{R}_v^2)$.
Hence, by Proposition \ref{propolocal},  the solution $f$ can be extended from  $[0,T^\star]$ to  $[0,T^\star+\delta]$ for some $\delta>0$  satisfying 
\begin{align*}
\|\rho\|_{1+a,T^*+\delta}+\|U\|_{1+a,T^*+\delta}\leq 2(\|\rho\|_{1+a,T^*}+\|U\|_{1+a,T^*}),
\end{align*}
since \eqref{zzz77}. 
Combining this with \eqref{zzz10} to get that 
\begin{equation}\|\rho\|_{1+a,T^\star+\delta}+	\|U\|_{1+a,T^\star+\delta}\leq 2C_1|||f_0|||_{1+a}<M|||f_0|||_{1+a},
\end{equation}
this contradicts to $T^\star<\infty$.
Therefore, $T^\star=\infty$ and  the proof is complete. 
		\end{proof}
\section{Pointwise estimates of characteristics}
In this section, we will prove the following estimates of characteristics that it will be used in the proof of Proposition \ref{propoglobal} and \ref{propolocal}. 

\begin{proposition}\label{estimates about Y and W} Let $(Y_{s,t},W_{s,t})$ be in \eqref{definition Y W}. Let $a\in (0,1)$, then there exists $\varepsilon_0\in(0,1)$ such that  for any $\|\mathfrak{g}\|_{1+a,T}\leq \varepsilon_0$, we have the following estimates:
	\begin{align}\nonumber
		&\sup_{0\leq s\leq t\leq T}\left(	\langle s\rangle\|Y_{s,t}\|_{L^\infty_{x,v}}+\langle s\rangle^{1+a}\|\nabla_x Y_{s,t}\|_{L^\infty_{x,v}}+\langle s\rangle^{1+a}\sup_{\alpha}	\frac{	\|\delta_{\alpha}^x	\nabla_xY_{s,t}\|_{L^\infty_{x,v}}}{|\alpha|^a}\right)\\&\quad\quad\quad\quad\quad\quad\quad\quad\quad\quad\quad\quad\quad\quad\quad\quad\quad\quad\quad\quad\quad+	\sup_{0\leq s\leq t\leq T}\langle s\rangle^{a}\|\nabla_v Y_{s,t}\|_{L^\infty_{x,v}}\lesssim_a \|\mathfrak{g}\|_{1+a,T},\label{results of Y small}
		\\& \nonumber
		\sup_{0\leq s\leq t\leq T}\left( \langle s\rangle^{2}	\|W_{s,t}\|_{L^\infty_{x,v}}+\langle s\rangle^{2+a}\|\nabla_x W_{s,t}\|_{L^\infty_{x,v}}+\langle s\rangle^{2+a}\sup_{\alpha}	\frac{	\|\delta_{\alpha}^x	\nabla_xW_{s,t}\|_{L^\infty_{x,v}}}{|\alpha|^a}\right)\\&\quad\quad\quad+	\sup_{0\leq s\leq t\leq T}\langle s\rangle^{1+a}\|\nabla_v W_{s,t}(x,v)\|_{L^\infty_{x,v}}+	\sup_{0\leq s\leq t\leq T}\langle s\rangle\sup_{\alpha}	\frac{	\|\delta_{\alpha}^v	\nabla_vW_{s,t}\|_{L^\infty_{x,v}}}{|\alpha|^a}\lesssim_a \|\mathfrak{g}\|_{1+a,T},\label{results of W small}
	\end{align}
	and
	\begin{align}	\sup_{0\leq s\leq t\leq T}\sup_{\alpha}	\frac{	\|\delta_{\alpha}^v	\nabla_vY_{s,t}\|_{L^\infty_{x,v}}}{|\alpha|^{a}}\lesssim_a \|\mathfrak{g}\|_{1+a,T}.\label{highestresults of Y small}
	\end{align}
	Moreover, for any $0\leq s\leq t< T$, we have a $C^{1}$ map~
	$
	(x,v)\mapsto\Psi_{s,t}(x,v),
	$
	which satisfies for all $x,v\in \mathbb{R}^2$:
	\begin{equation}\label{X(x,Psi)}
		X_{s,t}(x,\Psi_{s,t}(x,v))=x-(t-s)v,\qquad	
	\end{equation}
	and 
	\begin{equation}\label{z3}
		\langle s\rangle^{2}	\left|	\Psi_{s,t}(x,v)-v\right|+	\langle s\rangle^{2+a}	\left|	\nabla_x\Psi_{s,t}(x,v)\right|+\langle s\rangle^{1+a}	\left|	\nabla_v\left(\Psi_{s,t}(x,v)-v\right)\right|\lesssim_a \|\mathfrak{g}\|_{1+a,T}.
	\end{equation}
	Here we use the notations 
	\begin{align*}
		\delta_{\alpha}^xY_{s,t}(x,v)=Y_{s,t}(x,v)-Y_{s,t}(x-\alpha,v),~~~\delta_{\alpha}^vY_{s,t}(x,v)=Y_{s,t}(x,v)-Y_{s,t}(x,v-\alpha).
	\end{align*}
\end{proposition}
Clearly, Proposition \ref{estimates about Y and Wxx} follows from \eqref{results of W small} and \eqref{highestresults of Y small}.\vspace{0.15cm}\\
	\begin{proof}[Proof of Proposition \ref{estimates about Y and W}]
Thanks to \cite[Lemma 4.2]{HNX1} we have for any $\tau\in [0,T]
$, 			\begin{equation}\label{E estimate by rho}
	\langle \tau\rangle^{3}	\|\mathbf{E}(\tau)\|_{L^\infty}+	\langle \tau\rangle^{3+a}	\|\nabla_x \mathbf{E}(\tau)\|_{L^\infty}	+\langle \tau\rangle^{3+a}\|\nabla_x^2 \mathbf{E}(\tau)\|_{L^\infty}\leq c_0\|\mathfrak{g}\|_{1+a,T}.
\end{equation} 
Using the same argument as \cite[Proof of Proposition 4.1]{HNX1} and \eqref{E estimate by rho}, we obtain 	\eqref{results of Y small} and \eqref{results of W small}. Now, we prove \eqref{highestresults of Y small}.  One has
\begin{align*}
	&	\sup_{\alpha}	\frac{	\|\delta_{\alpha}^v	\nabla_vY_{s,t}\|_{L^\infty_{x,v}}}{|\alpha|^{a}}\leq 	\sup_{\alpha}\frac{1}{|\alpha|^{a}} \left(\int_{s}^{t}(\tau-s)\|\nabla_x \mathbf{E}(\tau)\|_{L^{\infty}}\|\delta_{\alpha}^v\nabla_v Y_{\tau,t}\|_{L^\infty_{x,v}}d\tau\right.\\&
	\left.+\int_{s}^{t}(\tau-s)\min\left\{\|\nabla_x \mathbf{E}(\tau)\|_{L^\infty},\|\nabla_x^2 \mathbf{E}(\tau)\|_{L^\infty} \left(|\alpha\|\tau|+\|\delta_\alpha^v Y_{\tau,t}\|_{L^\infty_{x,v}}\right)\right\}\left(|\tau|+\|\nabla_v Y_{\tau,t}\|_{L^\infty_{x,v}}\right)d\tau\right).
\end{align*}
By \eqref{E estimate by rho} and \eqref{results of Y small}, one gets
\begin{align*}
	\sup_{\alpha}	\frac{	\|\delta_{\alpha}^v	\nabla_vY_{s,t}\|_{L^\infty_{x,v}}}{|\alpha|^{a}}&\leq 	c_0 \int_{s}^{t}\frac{\tau-s}{\langle \tau\rangle^{3}}d\tau \|\mathfrak{g}\|_{1+a,T} \sup_{s\leq \tau\leq t}	\sup_{\alpha}	\frac{	\|\delta_{\alpha}^v	\nabla_vY_{s,t}\|_{L^\infty_{x,v}}}{|\alpha|^{a}}\\&\quad
	+c_0\|\mathfrak{g}\|_{1+a,T}\sup_{\alpha}\frac{1}{|\alpha|^{a}} \int_{s}^{t}(\tau-s)\min\left\{\frac{1}{\langle\tau\rangle^{3+a}},\frac{|\alpha|}{\langle\tau\rangle^{2+a}}\right\}\langle\tau\rangle d\tau\\&\leq \frac{1}{10}\sup_{s\leq \tau\leq t}	\sup_{\alpha}	\frac{	\|\delta_{\alpha}^v	\nabla_vY_{s,t}\|_{L^\infty_{x,v}}}{|\alpha|^{a}}+c_0\|\mathfrak{g}\|_{1+a,T}\sup_{\alpha}\frac{1}{|\alpha|^{a}}\int_{0}^{\infty}\frac{\min\{1,|\alpha|\langle\tau\rangle\}}{\langle\tau\rangle^{1+a}}d\tau
	\\&\leq \frac{1}{10}\sup_{s\leq \tau\leq t}	\sup_{\alpha}	\frac{	\|\delta_{\alpha}^v	\nabla_vY_{s,t}\|_{L^\infty_{x,v}}}{|\alpha|^{a}}+c_0'\|\mathfrak{g}\|_{1+a,T},
\end{align*}
for $\|\mathfrak{g}\|_{1+a,T}\leq \varepsilon_0$ small enough.	This implies \eqref{highestresults of Y small}. \\
Moreover, similar to \cite[Proof of Proposition 4.7]{HNX1}, we can construct a $C^{1}$ map~
$
(x,v)\in \mathbb{R}_x^2\times\mathbb{R}_v^2\mapsto\Psi_{s,t}(x,v),
$
satisfying \eqref{X(x,Psi)} and \eqref{z3}.  The proof of Proposition \ref{estimates about Y and W} is complete.
\end{proof}
		\section{Contribution of the initial data and the reaction term}\label{Contribution of the initial data} 
	 As a consequence of  Proposition \ref{estimates about Y and W}, we have  
		\begin{lemma}\label{I a estimate} 	Let  $h=h(x,v):\mathbb{R}^2_x\times\mathbb{R}^2_v\to \mathbb{R}$, define
			\begin{equation}
				\mathcal{I}_h(\mathfrak{g})(t,x)= \int_{\mathbb{R}^2}h(X_{0,t}(x,v),V_{0,t}(x,v))dv,
			\end{equation}
			where $(X_{s,t},V_{s,t})$ are in  Proposition \ref{estimates about Y and W} with $\|\mathfrak{g}\|_{1+a,T}\leq \varepsilon_0$. Then there holds
			\begin{equation}\label{I norm}
				\|\mathcal{I}_h(\mathfrak{g})\|_{1+a,T}\lesssim_a |||h|||_{1+a}.
			\end{equation}
		\end{lemma}
	The proof of Lemma \ref{I a estimate} is very similar to \cite[Proof of Proposition 5.1]{HNX1}, we omit it. \vspace{0.3cm}\\
	We next turn to the reaction term.  For any given $F:[0,T]\times \mathbb{R}^2\to \mathbb{R}$ and $\eta:  \mathbb{R}^2\to \mathbb{R}$, we denote 
		\begin{align*}
			\mathcal{T}[F,\eta](t,x)=-	\mathcal{T}_{NL}[F,\eta](t,x)+	\mathcal{T}_{L}[F,\eta](t,x),
		\end{align*}
		with 
		\begin{align*}
			&	\mathcal{T}_{L}[F,\eta](t,x)=\int_{0}^{t}\int_{\mathbb{R}^2}F(s,x-(t-s)v)\eta(v)dvds,\\&
			\mathcal{T}_{NL}[F,\eta](t,x)=\int_{0}^{t}\int_{\mathbb{R}^2}F(s,X_{s,t}(x,v))\eta(V_{s,t}(x,v))dvds,
		\end{align*}
		where $(X_{s,t},V_{s,t})$ are in  Proposition \ref{estimates about Y and W} with $\|\mathfrak{g}\|_{1+a,T}\leq \varepsilon_0$.
		By changing variable, we reformulate it as follows: 
		\begin{align*}
			&	\mathcal{T}_{L}[F,\eta](t,x)=	\int_{0}^{t}\int_{\mathbb{R}^2}F(s,w+\frac{s}{t}(x-w))\eta(\frac{x-w}{t})\frac{dwds}{t^2},\\&	\mathcal{T}_{NL}[F,\eta](t,x)=\int_{0}^{t}\int_{\mathbb{R}^2}
			F\left(s,Y_{s,t}(w,\frac{x-w}{t})+w+\frac{s(x-w)}{t}\right)\eta\left(W_{s,t}(w,\frac{x-w}{t})+\frac{x-w}{t}\right)
			\frac{dwds}{t^2}.
		\end{align*}
		We have the following Lemma.
		\begin{lemma} \label{esT} Let $a\in (0,1)$. Let $\eta$ be such that \begin{equation}\label{z9}
				\sum_{j=0}^3\langle v\rangle^3|\nabla^j\eta(v)|\leq 1.
			\end{equation}
 Then, 
			\\ \textbf{(1)}
			\begin{align}\label{z10'}
				\left|	\mathcal{T}[F,\eta](t,x)\right|\lesssim_{a} \|\mathfrak{g}\|_{1+a,T} \int_{0}^{t}\int_{\mathbb{R}^2}|F(s,x-(t-s)v)|\frac{1}{\langle s\rangle^{1+a}}\frac{dvds}{\langle v\rangle^3}.
			\end{align}
			In particular, for $p=1,\infty$,
			\begin{align}\label{z21}
				\langle t\rangle^\frac{2(p-1)}{p}\|	\mathcal{T}[F,\eta](t)\|_{L^p}\lesssim_{a} \|\mathfrak{g}\|_{1+a,T}\left(\sup_{s\in [0,t]}\langle s\rangle^\frac{2(p-1)}{p}\|F(s)\|_{L^p}+ \sup_{s\in [0,t]}\|F(s)\|_{L^1}\right).
			\end{align}
			\\ \textbf{(2)}
			For any	$0\leq t\leq T$, 
			\begin{align}
				&\sum_{p=1,\infty}\langle t\rangle^{\frac{2(p-1)}{p}+a}\|\mathcal{T}[F,\eta](t)\|_{\dot{B}^a_{p,\infty}}\lesssim_{a} \|\mathfrak{g}\|_{1+a,T} ^a\sum_{p=1,\infty}\sup_{s\in [0,t]} \langle s\rangle^{1+\frac{2(p-1)}{p}}  \left(\|F(s)\|_{L^p}+	\|F(s)\|_{\dot{F}^a_{p,\infty}}\right). \label{z12}
			\end{align}
		 \textbf{(3)}
	For any	$0\leq t\leq T$, 
	\begin{align}\nonumber
		&\sum_{p=1,\infty}\langle t\rangle^{\frac{2(p-1)}{p}+1+a}\|\nabla\mathcal{T}[F,\eta](t)\|_{\dot{B}^a_{p,\infty}}\\&\quad\quad\lesssim_{a}||\mathfrak{g}||_{1+a,T}^a\sum_{j=0,1}\sum_{p=1,\infty}\sup_{s\in [0,t]}\langle s\rangle^{1+\frac{2(p-1)}{p}}\left(\langle s\rangle^{\frac{a}{2}}\|\nabla^jF(s)\|_{\dot{F}^a_{p,\infty}}+\| F(s)\|_{L^p}\right), \label{z12b}
	\end{align}
where  $$\|g\|_{\dot{F}_{p,\infty}^{s}}:=\big\|\sup_{\alpha}\frac{|\delta_\alpha g(x)|}{|\alpha|^s}\big\|_{L^{p}_x}.$$
\end{lemma}
		\begin{proof} \textbf{ Step 1)}
			First of all, we change of variable and obtain
			\begin{align*}
				|\mathcal{T}[F,\eta](t,x)|=&\left|-\int_{0}^{t}\int_{\mathbb{R}^2}F(s,x-(t-s)v)\eta(V_{s,t}(x,\Psi_{s,t}(t,x)))|\det(\nabla_v\Psi_{s,t}(x,v))|dvds\right.\\
				&\quad\left.+\int_{0}^{t}\int_{\mathbb{R}^2}F(s,x-(t-s)v)\eta(v)dvds\right|\\
				\leq& \int_{0}^{t}\int_{\mathbb{R}^2}|F(s,x-(t-s)v)|\left|\eta(V_{s,t}(x,\Psi_{s,t}(t,x)))-\eta(v)\right|dvds\\
				&+\int_{0}^{t}\int_{\mathbb{R}^2}|F(s,x-(t-s)v)|\Big||\det(\nabla_v\Psi_{s,t}(x,v))|-1\Big|\left|\eta(V_{s,t}(x,\Psi_{s,t}(t,x)))\right|dvds.
			\end{align*}
			Hence, since we have 
			\begin{align*}
			&	\left|\eta(V_{s,t}(x,\Psi_{s,t}(x,v)))-\eta(v)\right|\leq\int_{0}^{1}\left|\nabla\eta(\varpi V_{s,t}(x,v)+(1-\varpi)v)\right|V_{s,t}(x,v)-v|d\varpi\\
				&\lesssim\sup_{\varpi\in[0,1]}\frac{|V_{s,t}(x,v)-v|}{\langle \varpi W_{s,t}(x-vt,v)+v\rangle^3}
				\sim\frac{|W_{s,t}(x-vt,v)|}{\langle v\rangle^3}\overset{\eqref{results of W small}}{\lesssim_{a}} \frac{\|\mathfrak{g}\|_{1+a,T}}{\langle s\rangle^{{2}}{\langle v\rangle^3}},
			\end{align*}
			and
			$\left|\eta(V_{s,t}(x,\Psi_{s,t}(t,x)))\right|\lesssim{\langle W_{s,t}(x-t\Psi_{s,t}(t,x),\Psi_{s,t}(t,x))+v\rangle^{-3}}\lesssim{\langle v\rangle^{-3}}$, 
			where we use the fact that 
			$ 	\sup_{0\leq s\leq t}\|W_{s,t}\|_{L^\infty_{x,v}}\leq 1 $ in \eqref{results of W small}. Hence
			\begin{align*}
				\left|	\mathcal{T}[F,\eta](t,x)\right|
				\overset{\eqref{z3}}{\lesssim_{a}}  \|\mathfrak{g}\|_{1+a,T}\int_{0}^{t}\int_{\mathbb{R}^2}|F(s,x-(t-s)v)|\frac{1}{\langle s\rangle^{1+a}}\frac{dvds}{\langle v\rangle^3}.
			\end{align*}
			This gives \eqref{z10'}.  Moreover,	\begin{align*}
				\|	\mathcal{T}[F,\eta](t)\|_{L^p}\lesssim_{a} \|\mathfrak{g}\|_{1+a,T} \int_{0}^{t}\left\|\int_{\mathbb{R}^2}|F(s,x-(t-s)v)|\frac{dv}{\langle v\rangle^3}\right\|_{L^p_x}\frac{ds}{\langle s\rangle^{1+a}}.
			\end{align*} 
			Then, thanks to Lemma \ref{mathcal H1H2}, one gets
			\begin{align}\nonumber
				\|	\mathcal{T}[F,\eta](t)\|_{L^p}\lesssim_{a}& t^{-\frac{2(p-1)}{p}} \|\mathfrak{g}\|_{1+a,T} \int_{0}^{t/2}\|F(s)\|_{L^1}\frac{ds}{\langle s\rangle^{1+a}}+\|\mathfrak{g}\|_{1+a,T}\int_{t/2}^{t}\|F(s)\|_{L^p}\frac{ds}{\langle s\rangle^{1+a}}\\
				\lesssim_{a}& \frac{1}{t^\frac{2(p-1)}{p}}\|\mathfrak{g}\|_{1+a,T} \left(\sup_{s\in [0,t]}\|F(s)\|_{L^1}+ \sup_{s\in [0,t]}\langle s\rangle^\frac{2(p-1)}{p}\|F(s)\|_{L^p}\right).\label{zz}
			\end{align}
			Combining this with 
			\begin{align*}
				\|	\mathcal{T}[F,\eta](t)\|_{L^p}\lesssim_{a} \|\mathfrak{g}\|_{1+a,T} \int_{0}^{t}\|F(s)\|_{L^p}\frac{1}{\langle s\rangle^{1+a}}ds\lesssim_{a} \|\mathfrak{g}\|_{1+a,T}\sup_{s\in [0,t]}\|F(s)\|_{L^p},
			\end{align*} to obtain \eqref{z21}.\vspace{0.3cm}\\
				Note that  we only use $\sup_v\langle v\rangle^3(|\eta(v)|+|\nabla \eta(v)|)\leq 1$ in the proof of \eqref{z10'}. \vspace{0.3cm}\\	
			\textbf{ Step 2)} We estimate $\|\mathcal{T}[F,\eta](t)\|_{\dot{B}^a_{p,\infty}}.$ \vspace{0.15cm}\\
			\textbf{2.1)} Case $0\leq t\leq \min\{1,T\}$. We divide $\delta_{\alpha}\mathcal{T}[F,\eta](t,x)$ into three terms:
			\begin{align*}
				\delta_{\alpha}\mathcal{T}[F,\eta](t,x)=\overline{\mathcal{T}}_{\alpha}^1[F,\eta](t,x)|+\overline{\mathcal{T}}_{\alpha}^2[F,\eta](t,x)+	\overline{\mathcal{T}}_{\alpha}^3[F,\eta](t,x)
			\end{align*}
			where
			\begin{align*}
				&	\overline{\mathcal{T}}_{\alpha}^1[F,\eta](t,x)=\int_{0}^{t}\int_{\mathbb{R}^2}
		\left(	-\delta_{\alpha}F(s,.)\left(X_{s,t}(x,v)\right)\eta\left(V_{s,t}(x,v)\right)	+\delta_{\alpha}F(s,.)(x-(t-s)v)\eta(v)\right)dvds,
				\\	&\overline{\mathcal{T}_{\alpha}}^2[F,\eta](t,x)=\int_{0}^{t}\int_{\mathbb{R}^2}
				\left[F\left(s,X_{s,t}(x-\alpha,v)\right)-F\left(s,,X_{s,t}(x,v)-\alpha\right) \right]\eta\left(V_{s,t}(x,v)\right)dvds,
				\\
				&	\overline{\mathcal{T}}_{\alpha}^3[F,\eta](t,x)=\int_{0}^{t}\int_{\mathbb{R}^2}
				F\left(s,X_{s,t}(x-\alpha,v)\right) \left[\eta\left(V_{s,t}(x-\alpha,v)\right)-\eta\left(V_{s,t}(x,v)\right) \right]
				dvds.
			\end{align*}
			For the first term, using \eqref{z9} and \eqref{z10'} we have
			\begin{align*}
				\|\overline{\mathcal{T}}_{\alpha}^1[F,\eta](t)\|_{L^p}&=\|\mathcal{T}[\delta_{\alpha}F,\eta](t)\|_{L^p}\lesssim_{a} \|\mathfrak{g}\|_{1+a,T}\left(\sup_{s\in [0,t]}\|\delta_{\alpha}F(s)\|_{L^1}+\sup_{s\in [0,t]}\|\delta_{\alpha}F(s)\|_{L^p} \right)\\
				&	\lesssim_{a}|\alpha|^a \|\mathfrak{g}\|_{1+a,T} \left(\sup_{s\in [0,t]}\|F(s)\|_{\dot{B}^a_{1,\infty}}+
				\sup_{s\in [0,t]}\|F(s)\|_{\dot{B}^a_{p,\infty}}\right).
			\end{align*}
			Since 
			\begin{align*}
				\left|X_{s,t}(x-\alpha,v)-\left(X_{s,t}(x,v)-\alpha\right)\right|=\left|Y_{s,t}(x-\alpha-tv,v)-Y_{s,t}(x-tv,v)\right|\leq|\alpha|\|\nabla_x Y_{s,t}\|_{L^\infty},
			\end{align*} 
			we can estimate  $	\overline{\mathcal{T}}_{\alpha}^2[F,\eta](t,x)$ as follows
			\begin{align*}
				\left\|\overline{\mathcal{T}}_{\alpha}^2[F,\eta](t)\right\|_{L^p}
				\lesssim&|\alpha|^a\int_{0}^{t}\int_{\mathbb{R}^2}\|\sup_z	|z|^{-a}|\delta_{z}F(s,.)\|_{L^p}\|\nabla_x Y_{{s,t}}\|_{L^\infty_{x,v}}^a\frac{dv}{\langle v \rangle^3}ds
				\\
				\lesssim_{a}&|\alpha|^a\|\mathfrak{g}\|_{1+a,T}^a\left(\sup_{s\in [0,t]}\langle s\rangle\|F(s)\|_{\dot{F}^a_{p,\infty}}\right)
				\int_{0}^{t}\langle s\rangle^{-1-(1+a)a}ds
			\\\lesssim_{a}&|\alpha|^a\|\mathfrak{g}\|_{1+a,T}^a\sup_{s\in [0,t]}\langle s\rangle\|F(s)\|_{\dot{F}^a_{p,\infty}}.
			\end{align*}
			Note that by \eqref{results of W small}, one has
			$
			\sup_{s,t}\|W_{s,t}\|_{L^\infty_{x,v}}\leq 1.
			$
			Then,
			\begin{align}\label{V sim v}
				\langle \varpi V_{s,t}(x,v)+(1-\varpi)V_{s,t}(x-\alpha,v)\rangle= \langle v+\varpi W_{s,t}(x-tv,v)+(1-\varpi)W_{s,t}(x-\alpha-tv,v)\rangle
				\sim\langle v\rangle.
			\end{align}	
			We can estimate $\overline{\mathcal{T}}_{\alpha}^3[F,\eta](t)$,
			\begin{align*}				
				\|\overline{\mathcal{T}}_{\alpha}^3[F,\eta](t)\|_{L^p}&\leq \int_{0}^{t}\int_{\mathbb{R}^2}\int_0^1\|F(s)\|_{L^p}\left|\nabla\eta(\varpi V_{s,t}(x,v)+(1-\varpi)V_{s,t}(x-\alpha,v))\right| \left|\delta_\alpha^xV_{s,t}(x,v)\right| d\varpi dvds\\
				&\lesssim |\alpha|^a\int_{0}^{t}\int_{\mathbb{R}^2}\|F(s)\|_{L^p}\frac{1}{\langle v\rangle^3}\sup_\alpha\frac{\left|\delta_\alpha^xV_{s,t}(x,v)\right|}{|\alpha|^a} dvds\\
				&\overset{\eqref{results of W small}}{\lesssim_{a}}|\alpha|^a\|\mathfrak{g}\|_{1+a,T}\left(\sup_{s\in [0,t]}\|F(s)\|_{L^p}\right)\int_0^t \frac{ds}{\langle s\rangle^{2+a}}
				\lesssim_{a}|\alpha|^a\|\mathfrak{g}\|_{1+a,T} \sup_{s\in [0,t]} \|F(s)\|_{L^p}.
			\end{align*}
			In conclusion,
			for $0\leq t\leq \min\{1,T\}$,
			\begin{align}\label{d3}
				\|\mathcal{T}[F,\eta](t)\|_{\dot{B}^a_{p,\infty}} 
				\lesssim_{a} \|\mathfrak{g}\|_{1+a,T} ^a \left(\sup_{s\in [0,t]} \|F(s)\|_{\dot{F}^a_{p,\infty}}
				+\sup_{s\in [0,t]}\|F(s)\|_{L^p}\right).
	\end{align}
	\textbf{2.2)} Case $T> 1$ and $1\leq t\leq T$. 
	Set 
	\begin{align}
	\nonumber
			&Z_1(x)=Y_{s,t}(w,\frac{x-w}{t})+w+\frac{s(x-w)}{t},~	Z_2(x)=W_{s,t}(w,\frac{x-w}{t})+\frac{x-w}{t},\\
			&Z_3(x)=W_{s,t}(w,\frac{x-\alpha-w}{t})+\frac{x-w}{t},~~
			Z_4(x)=Y_{s,t}(w,\frac{x-w}{t})+w+\frac{s(x-\alpha-w)}{t}.\label{Z1234}
	\end{align}
	We have 
	\begin{align*}
		\delta_{\alpha}\mathcal{T}[F,\eta](t,x)=\mathcal{T}_{\alpha}^1[F,\eta](t,x)|+\mathcal{T}_{\alpha}^2[F,\eta](t,x)+	\mathcal{T}_{\alpha}^3[F,\eta](t,x)+	\mathcal{T}_{\alpha}^4[F,\eta](t,x),
	\end{align*}
	where 
	\begin{align*}
		&	\mathcal{T}_{\alpha}^1[F,\eta](t,x)=-\int_{0}^{t}\int_{\mathbb{R}^2}
				\delta_{\frac{s\alpha}{t}}F(s,.)\left(Z_1(x)\right)\eta\left(Z_2(x)\right)
				\frac{dwds}{t^{2}} \\&\quad\quad+	\int_{0}^{t}\int_{\mathbb{R}^2}	\delta_{\frac{s\alpha}{t}}F(s,.)(w+\frac{s}{t}(x-w))\eta(\frac{x-w}{t})\frac{dwds}{t^2},\\
				&	\mathcal{T}_{\alpha}^2[F,\eta](t,x)=\int_{0}^{t}\int_{\mathbb{R}^2}
				F\left(s,Z_1(x-\alpha)\right)(\delta_{\frac{\alpha}{t}}\eta)\left(Z_2(x)\right)
				\frac{dwds}{t^{2}}
			\\&\quad\quad	-	\int_{0}^{t}\int_{\mathbb{R}^2}F(s,w+\frac{s}{t}(x-w-\alpha))(\delta_{\frac{\alpha}{t}}\eta)(\frac{x-w}{t})\frac{dwds}{t^2},\\
				&	\mathcal{T}_{\alpha}^3[F,\eta](t,x)=\int_{0}^{t}\int_{\mathbb{R}^2}
				\left(F\left(s,Z_1(x-\alpha)\right)-F\left(s,Z_4(x)\right)\right) \eta\left(Z_2(x)\right)
				\frac{dwds}{t^{2}},\\
				&	\mathcal{T}_{\alpha}^4[F,\eta](t,x)=\int_{0}^{t}\int_{\mathbb{R}^2}
				F\left(s,Z_1(x-\alpha)\right) \left(\eta\left(Z_3(x)\right)-\eta\left(Z_2(x)\right)\right)
				\frac{dwds}{t^{2}}.
			\end{align*}
			Going back to the variable $v=\frac{x-w}{t}$,
			and denoting 
			\begin{align*}
				Z_5(x)=Y_{s,t}(x-tv,v-\frac{\alpha}{t})+x-(t-s)v-\frac{s\alpha}{t},~~
				Z_6(x)=W_{s,t}(x-tv,v-\frac{\alpha}{t})+v,
			\end{align*}
			we reformulate it as follows:
			\begin{align*}
				&\mathcal{T}_{\alpha}^1[F,\eta](t,x)=\int_{0}^{t}\int_{\mathbb{R}^2}
				\left(-	\delta_{\frac{s\alpha}{t}}F(s,.)\left(X_{s,t}(x,v)\right)\eta\left(V_{s,t}(x,v)\right)+	\delta_{\frac{s\alpha}{t}}F(s,.)(x-(t-s)v)\eta(v)\right)dvds,\\&
				\mathcal{T}_{\alpha}^2[F,\eta](t,x)=\int_{0}^{t}\int_{\mathbb{R}^2}
				\left(	F\left(s,X_{s,t}(x,v)\right)\left(\delta_{\frac{\alpha}{t}}\eta\right)\left(V_{s,t}(x,v)\right)-F(s,x-(t-s)v)\left(\delta_{\frac{\alpha}{t}}\eta\right)(v)\right)dvds,
				\\
				&	\mathcal{T}_{\alpha}^3[F,\eta](t,x)= \int_{0}^{t}\int_{\mathbb{R}^2}
				\left(F\left(s,X_{s,t}(x,v)-\frac{s\alpha}{t}\right)-F\left(s,Z_5(x)\right)\right)
				\eta\left(V_{s,t}(x,v)\right)dvds,
				\\&
				\mathcal{T}_{\alpha}^4[F,\eta](t,x)= \int_{0}^{t}\int_{\mathbb{R}^2}
				F\left(s,Z_5(x)\right)
				\left(\eta\left(V_{s,t}(x,v)\right)-\eta\left(Z_6(x)\right)\right)
				dvds.
			\end{align*}
		Note that for $p=1,\infty$ and $|\alpha|\geq t$,
		\begin{align*}
		\langle t\rangle^{\frac{2(p-1)}{p}+a}\frac{\|\delta_{\alpha}\mathcal{T}[F,\eta](t)\|_{L^p}}{|\alpha|^a}&\leq 2	\langle t\rangle^{\frac{2(p-1)}{p}}\|\delta_{\alpha}\mathcal{T}[F,\eta](t)\|_{L^p}\leq 4	\langle t\rangle^{\frac{2(p-1)}{p}}\|\mathcal{T}[F,\eta](t)\|_{L^p}\\& \overset{\eqref{z21}}\lesssim_{a} \|\mathfrak{g}\|_{1+a,T}\left(\sup_{s\in [0,t]}\langle s\rangle^\frac{2(p-1)}{p}\|F(s)\|_{L^p}+ \sup_{s\in [0,t]}\|F(s)\|_{L^1}\right).
		\end{align*}
	So, it is enough to consider $|\alpha|\leq t$.\\
			\textbf{2.2.1)} Estimate $	\mathcal{T}_{\alpha}^1[F,\eta](t,x)$.\\Applying  \eqref{z21} with $\delta_{\frac{s\alpha}{t}}F$, we have
			\begin{align*}
				\|\mathcal{T}_{\alpha}^1[F,\eta](t)\|_{L^p}&=\|\mathcal{T}[\delta_{\frac{s\alpha}{t}}F,\eta](t)\|_{L^p}\\&\lesssim_{a} t^{-\frac{2(p-1)}{p}}\|\mathfrak{g}\|_{1+a,T}  \left(\sup_{s\in [0,t]}\|\delta_{\frac{s\alpha}{t}}F(s)\|_{L^1}+\sup_{s\in [0,t]}\langle s\rangle^\frac{2(p-1)}{p}\|\delta_{\frac{s\alpha}{t}}F(s)\|_{L^p} \right)\\
				&	\lesssim_{a}|\alpha|^a t^{-\frac{2(p-1)}{p}-a}\|\mathfrak{g}\|_{1+a,T}  \left(\sup_{s\in [0,t]}\langle s\rangle^a\|F(s)\|_{\dot{B}^a_{1,\infty}}+
				\sup_{s\in [0,t]}\langle s\rangle^{\frac{2(p-1)}{p}+a}\|F(s)\|_{\dot{B}^a_{p,\infty}}\right).
			\end{align*}
			Thus, we can yield
			\begin{equation}\label{T alpha lower 1}
				t^{\frac{2(p-1)}{p}+a}\sup_\alpha	\frac{	\|\mathcal{T}_{\alpha}^1[F,\eta](t)\|_{L^p}}{|\alpha|^a}\lesssim_{a} \|\mathfrak{g}\|_{1+a,T} \left(\sup_{s\in [0,t]}\langle s\rangle^a\|F(s)\|_{\dot{B}^a_{1,\infty}}+
				\sup_{s\in [0,t]}\langle s\rangle^{\frac{2(p-1)}{p}+a}\|F(s)\|_{\dot{B}^a_{p,\infty}}\right).
			\end{equation}
			\textbf{2.2.2)} Estimate $	\mathcal{T}_{\alpha}^2[F,\eta](t,x)$.  \\Note that 
			\begin{align*}
				|\delta_{\frac{\alpha}{t}}\eta(v)|+|\nabla_v(\delta_{\frac{\alpha}{t}}\eta)(v)|\lesssim \min\left\{\frac{|\alpha|}{t},1\right\}\left(\frac{1}{\langle v\rangle^3}+\frac{1}{\langle v-\frac{\alpha}{t}\rangle^3}\right)\lesssim\frac{|\alpha|^a}{t^a}\frac{1}{\langle v\rangle^3},
			\end{align*}
		since $|\alpha|\leq t$. 
			Thus, by Step 1 we have
			\begin{align*}
				\|\mathcal{T}_{\alpha}^2[F,\eta](t)\|_{L^p}&=\frac{|\alpha|^a}{t^a}\left\|\mathcal{T}\left[F,\frac{t^a}{|\alpha|^a}\delta_{\frac{\alpha}{t}}\eta(v)\right]\right\|_{L^p}\\&\lesssim_{a}\frac{|\alpha|^a}{t^{a+\frac{2(p-1)}{p}}}\|\mathfrak{g}\|_{1+a,T} \left(\sup_{s\in [0,t]}\|F(s)\|_{L^1}+\sup_{s\in [0,t]}\langle s\rangle^\frac{2(p-1)}{p}\|F(s)\|_{L^p} \right),
			\end{align*}
			which implies 
			\begin{align*}
				t^{\frac{2(p-1)}{p}+a}\sup_\alpha	\frac{	\|\mathcal{T}_{\alpha}^2[F,\eta](t)\|_{L^p}}{|\alpha|^a}\lesssim_{a} \|\mathfrak{g}\|_{1+a,T} \left(\sup_{s\in [0,t]}\langle s\rangle^\frac{2(p-1)}{p}\|F(s)\|_{L^p}+\sup_{s\in [0,t]}\|F(s)\|_{L^1}\right).
			\end{align*}
			\textbf{2.2.3)} Estimate $	\mathcal{T}_{\alpha}^3[F,\eta](t,x)$ .
			\begin{align}\nonumber
				&	|\mathcal{T}_{\alpha}^3[F,\eta](t,x)|\\\nonumber
				&\lesssim_{\mathbf{c}_1}\int_{0}^{t}\int_{\mathbb{R}^2}
				\left|F\left(s,X_{s,t}(x,v)-\frac{s\alpha}{t}\right)-F\left(s,Z_5(x)\right)\right|\frac{dvds}{\langle v\rangle^3}\\\nonumber
				&\lesssim_{\mathbf{c}_1}\int_{0}^{t}\int_{\mathbb{R}^2}
				\left(	\sup_z	|z|^{-a}\left|\delta_{z}F(s,.)\left(Y_{s,t}(x-tv,v)+x-(t-s)v-\frac{s\alpha}{t}\right)\right|\right)\\&\nonumber\quad\quad\times\left|Y_{s,t}(x-tv,v)-Y_{s,t}(x-tv,v-\frac{\alpha}{t})\right|^a\frac{dvds}{\langle v\rangle^3}\\\label{T3fml 1}
				&\lesssim_{a}\frac{|\alpha|^a}{t^a} \|\mathfrak{g}\|_{1+a,T} ^a\int_{0}^{t}\int_{\mathbb{R}^2}
				\sup_z	|z|^{-a}\left|\delta_{z}F(s,.)\left(Y_{s,t}(x-tv,v)+x-(t-s)v-\frac{s\alpha}{t}\right)\right|\frac{ds}{\langle s\rangle^{a^2}} \frac{dv}{\langle v\rangle^3}.
			\end{align}
			Thus, as \eqref{zz}, we apply Lemma \ref{mathcal H1H2} with $\mathcal{H}=\sup_z	|z|^{-a}|\delta_{z}F(s,.)|$ and $\varphi(x,v)=Y_{s,t}(x-tv,v-\frac{\alpha}{t})-\frac{s\alpha}{t}$,
			\begin{align*}
				&\|\mathcal{T}_{\alpha}^3[F,\eta](t)\|_{L^p}\lesssim _{a}  \underbrace{\frac{|\alpha|^a}{t^{\frac{2(p-1)}{p}+a}}\|\mathfrak{g}\|_{1+a,T} ^a\int_{0}^{t/2}
					\|F(s)\|_{\dot{F}^a_{1,\infty}}\frac{ds}{\langle s\rangle^{a^2}} }_{\text{Thanks to }\eqref{mathcal H 1}}+\underbrace{\frac{|\alpha|^a}{t^a}\|\mathfrak{g}\|_{1+a,T} ^a\int_{t/2}^{t}
					\|F(s)\|_{\dot{F}^a_{p,\infty}}\frac{ds}{\langle s\rangle^{a^2}} }_{\text{Thanks to }\eqref{mathcal H p}}\\&\quad\quad\quad
				\lesssim  _{a} \frac{|\alpha|^a}{t^{\frac{2(p-1)}{p}+a}}\|\mathfrak{g}\|_{1+a,T} ^a\left(
				\sup_{s\in [0,t]}\langle s\rangle\|F(s)\|_{\dot{F}^a_{1,\infty}}+\sup_{s\in [0,t]} \langle s\rangle^{\frac{2(p-1)}{p}+1}	\|F(s)\|_{\dot{F}^a_{p,\infty}}\right)\int_{0}^{t}\frac{ds}{\langle s\rangle^{1+a^2}}\\ 
				&\quad\quad\quad\lesssim_{a}\frac{|\alpha|^a}{t^{\frac{2(p-1)}{p}+a}}\|\mathfrak{g}\|_{1+a,T} ^a\left(
				\sup_{s\in [0,t]}\langle s\rangle\|F(s)\|_{\dot{F}^a_{1,\infty}}+\sup_{s\in [0,t]} \langle s\rangle^{\frac{2(p-1)}{p}+1}	\|F(s)\|_{\dot{F}^a_{p,\infty}}\right).
			\end{align*}
			Thus, 
			\begin{align*}
				t^{\frac{2(p-1)}{p}+a}\sup_\alpha\frac{	\|\mathcal{T}_{\alpha}^3[F,\eta](t)\|_{L^p}}{|\alpha|^a}&\lesssim _{a} \|\mathfrak{g}\|_{1+a,T} ^a\left(
				\sup_{s\in [0,t]}\langle s\rangle\|F(s)\|_{\dot{F}^a_{1,\infty}}+\sup_{s\in [0,t]} \langle s\rangle^{\frac{2(p-1)}{p}+1}	\|F(s)\|_{\dot{F}^a_{p,\infty}}\right).
			\end{align*}
			\textbf{2.2.4)} Estimate $	\mathcal{T}_{\alpha}^4[F,\eta](t,x)$.\\  One has
			\begin{align}\nonumber
				&
				|	\mathcal{T}_{\alpha}^4[F,\eta](t,x)|\lesssim \int_{0}^{t}\int_{\mathbb{R}^2}
				\left|	F\left(s,Z_5(x)\right)\right| \left|W_{s,t}(x-tv,v)-W_{s,t}(x-tv,v-\frac{\alpha}{t})\right|
				\frac{	dvds}{\langle v\rangle^3}\\\nonumber
				&\lesssim_{a} \min\left\{1,\frac{|\alpha|}{t}\right\} \|\mathfrak{g}\|_{1+a,T} \int_{0}^{t}\int_{\mathbb{R}^2}
				\left|F\left(s,Y_{s,t}(x-tv,v-\frac{\alpha}{t})+x-(t-s)v-\frac{s\alpha}{t}\right)\right|  \frac{1}{\langle s\rangle^{1+a}}
				\frac{	dvds}{\langle v\rangle^3}\\\label{T4fml 1}
				&\lesssim_{a} \frac{|\alpha|^a}{\ t
					^a}   \|\mathfrak{g}\|_{1+a,T} \int_{0}^{t}\int_{\mathbb{R}^2}
				\left|F\left(s,Y_{s,t}(x-tv,v-\frac{\alpha}{t})+x-(t-s)v-\frac{s\alpha}{t}\right)\right|  \frac{dv}{\langle v\rangle^3}\frac{ds}{\langle s\rangle^{1+a}}.	
			\end{align}
			Then, as \eqref{zz} we  apply Lemma \ref{mathcal H1H2} with $\mathcal{H}=F$ and $\varphi(x,v)=Y_{s,t}(x-tv,v-\frac{\alpha}{t})-\frac{s\alpha}{t}$ to get
			\begin{align*}
				\|\mathcal{T}_{\alpha}^4[F,\eta](t)\|_{L^p}&\lesssim_{a} \underbrace{\frac{|\alpha|^a}{\ t
						^{a+\frac{2(p-1)}{p}}}  \|\mathfrak{g}\|_{1+a,T}  \int_{0}^{t/2}
					\|F(s)\|_{L^1}  \frac{ds}{\langle s\rangle^{1+a}}}_{\text{Thanks to }\eqref{mathcal H 1}}+ \underbrace{\frac{|\alpha|^{a}}{\ t
						^a}  \|\mathfrak{g}\|_{1+a,T}  \int_{t/2}^{t}
					\|F(s)\|_{L^p}  \frac{ds}{\langle s\rangle^{1+a}}}_{\text{Thanks to }\eqref{mathcal H p}} \\&\lesssim_{a} 
				\frac{|\alpha|^a}{t^{\frac{2(p-1)}{p}+a}}\|\mathfrak{g}\|_{1+a,T}   \left( \sup_{s\in [0,t]} \langle s\rangle^\frac{2(p-1)}{p} \|F(s)\|_{L^p}+\sup_{s\in [0,t]}\|F(s)\|_{L^1}\right).
			\end{align*}
			Hence, we obtain
			\begin{align*}
				t^{\frac{2(p-1)}{p}+a}\sup_\alpha\frac{	\|\mathcal{T}_{\alpha}^4[F,\eta](t)\|_{L^p}}{|\alpha|^a}&\lesssim_{a}  \|\mathfrak{g}\|_{1+a,T}\left(
				\sup_{s\in [0,t]}\langle s\rangle^a\|F(s)\|_{\dot{F}^a_{1,\infty}}+\sup_{s\in [0,t]} \langle s\rangle^{\frac{2(p-1)}{p}+a}	\|F(s)\|_{\dot{F}^a_{p,\infty}}\right).
			\end{align*}
			Summing up, we conclude that 
			\begin{align}\label{d2}
				&\sum_{p=1,\infty} t^{\frac{2(p-1)}{p}+a}\|\mathcal{T}[F,\eta](t)\|_{\dot{B}^a_{p,\infty}}\lesssim _{a} \|\mathfrak{g}\|_{1+a,T} ^a\sum_{p=1,\infty}\sup_{s\in [0,t]} \langle s\rangle^{1+\frac{2(p-1)}{p}}  \left(\|F(s)\|_{L^p}+	\|F(s)\|_{\dot{F}^a_{p,\infty}}\right). 
			\end{align}
			Hence, \eqref{z12} follows from \eqref{d3} and  \eqref{d2}. 
			\vspace{0.25cm}\\
			\textbf{Step 3)} We estimate $\|\nabla\mathcal{T}[F,\eta](t)\|_{\dot{B}^a_{p,\infty}}.$ 
			 \vspace{0.15cm}\\
			\textbf{3.1)} Case $0\leq t\leq \min\{1,T\}$.  We have 
			\begin{align*}
					\partial_{x_i}\mathcal{T}[F,\eta](t,x)=\mathcal{T}[\partial_{x_i}F,\eta](t,x)+\mathcal{T}^{R,i}[F,\eta](t,x),~~~i=1,2,
			\end{align*}
		where 
		\begin{align*}
			\mathcal{T}^{R,i}[F,\eta](t,x)&= \int_{0}^{t}\int_{\mathbb{R}^2}F(s,X_{s,t}(x,v))\nabla\eta(V_{s,t}(x,v)) \partial_{x_i}W_{s,t}(x-vt,v)dvds\\&
		+	\int_{0}^{t}\int_{\mathbb{R}^2}(\nabla F)(s,X_{s,t}(x,v)) \partial_{x_i}Y_{s,t}(x-vt,v) \eta(V_{s,t}(x,v))dvds.
		\end{align*}
By \eqref{d3}, one has 
	\begin{equation*}
		\sum_{p=1,\infty}\sum_{i=1,2}	\|	\mathcal{T}[\partial_{x_i}F,\eta](t)\|_{\dot{B}^a_{p,\infty}}\lesssim_a ||\mathfrak{g}||_{1+a,T}^a\sum_{p=1,\infty}\sup_{s\in [0,t]}	 \left( \|\nabla F(s)\|_{L^p}+	\|\nabla F(s)\|_{\dot{F}^a_{p,\infty}}\right).
	\end{equation*}
It is easy to check (see \cite[Proof of Proposition 6.6, the second part]{HNX1}) that 
\begin{equation*}
\sum_{p=1,\infty}\sum_{i=1,2}	\|	\mathcal{T}^{R,i}[F,\eta](t)\|_{\dot{B}^a_{p,\infty}}\lesssim_a ||\mathfrak{g}||_{1+a,T}\sum_{p=1,\infty}\sup_{s\in [0,t]}	 \left( \|F(s)\|_{L^p}+	\|\nabla F(s)\|_{\dot{F}^a_{p,\infty}}\right).
\end{equation*}
Thus, we obtain \eqref{z12b} with $0\leq t\leq \min\{1,T\}.$\\
			\textbf{3.2)} Case $T> 1$ and $1\leq t\leq T$.  We have 
				\begin{align*}
				\partial_{x_i}\mathcal{T}[F,\eta](t,x)=\frac{1}{t}	\mathcal{T}[	\tilde{F}_i,\eta](t,x)+\frac{1}{t} 	\mathcal{T}[F,\partial_{i}\eta](t,x)+\mathcal{T}^{1,i}[F,\eta](t,x)+\mathcal{T}^{1,i}[F,\eta](t,x),~~i=1,2,
			\end{align*}
			where $
			\tilde{F}_i(s,.):=s\partial_{x_i}F(s,.)$, 
			\begin{align*}
				\mathcal{T}^{1,i}[F,\eta](t,x):&=-\int_{0}^{t}\int_{\mathbb{R}^2}
				F\left(s,Y_{s,t}(w,\frac{x-w}{t})+w+\frac{s(x-w)}{t}\right)\\&\quad\quad\quad\times(\nabla\eta)\left(W_{s,t}(w,\frac{x-w}{t})+\frac{x-w}{t}\right).(\partial_{v_i}W_{s,t})(w,\frac{x-w}{t})
				\frac{dwds}{t^3},
			\end{align*}
			\begin{align*}
				\mathcal{T}^{2,i}[F,\eta](t,x):&=-\int_{0}^{t}\int_{\mathbb{R}^2}
				(\nabla	F)\left(s,Y_{s,t}(w,\frac{x-w}{t})+w+\frac{s(x-w)}{t}\right).(\partial_{v_i}Y_{s,t})(w,\frac{x-w}{t})\\&\quad\quad\quad\times\eta\left(W_{s,t}(w,\frac{x-w}{t})+\frac{x-w}{t}\right)
				\frac{dwds}{t^3}.
			\end{align*}
		Applying  \eqref{d2} to $	\mathcal{T}[	\tilde{F}_i,\eta]$ and $\mathcal{T}[F,\partial_{i}\eta]$ to get 
			\begin{align}\label{zz1}
			&\sum_{p=1,\infty} t^{\frac{2(p-1)}{p}+a}\|	\mathcal{T}[	\tilde{F}_i,\eta](t)\|_{\dot{B}^a_{p,\infty}}\lesssim _{a} \|\mathfrak{g}\|_{1+a,T} ^a\sum_{p=1,\infty}\sup_{s\in [0,t]}\langle s\rangle^{\frac{2(p-1)}{p}+1}  \left( \|\nabla F(s)\|_{L^p}+	\|\nabla F(s)\|_{\dot{F}^a_{p,\infty}}\right),\\&
			\sum_{p=1,\infty} t^{\frac{2(p-1)}{p}+a}\|	\mathcal{T}[F,\partial_{i}\eta](t)\|_{\dot{B}^a_{p,\infty}}\lesssim _{a} \|\mathfrak{g}\|_{1+a,T} ^a\sum_{p=1,\infty} \sup_{s\in [0,t]}\langle s\rangle^{\frac{2(p-1)}{p}}\left(  \| F(s)\|_{L^p}+	\| F(s)\|_{\dot{F}^a_{p,\infty}}\right).\label{zz2}
		\end{align}
Now we estimate  $\mathcal{T}^{1,i}[F,\eta](t,x)$ and $\mathcal{T}^{2,i}[F,\eta](t,x)$.  Indeed, \\
\textit{Case 1: }$|\alpha|\geq t$. Thanks to \eqref{results of Y small}, \eqref{results of W small} and \eqref{z9}, one has 
\begin{align*}
	&	||\delta_{\alpha}\mathcal{T}^{1,i}[F,\eta](t)||_{L^p}+	||\delta_{\alpha}\mathcal{T}^{2,i}[F,\eta](t)||_{L^p}\leq 	2||\mathcal{T}^{1,i}[F,\eta](t)||_{L^p}+	2||\mathcal{T}^{2,i}[F,\eta](t)||_{L^p} \\&\lesssim_{a}t^{-1}||\mathfrak{g}||_{1+a,T} \left\|\int_{0}^{t}\int_{\mathbb{R}^2}
	\left|F\left(s,Y_{s,t}(w,\frac{x-w}{t})+w+\frac{s(x-w)}{t}\right)\right|\frac{1}{\langle \frac{x-w}{t}\rangle^3}
	\frac{dwds}{\langle s\rangle^{1+a}t^2}\right\|_{L^p_x}\\&\quad+t^{-1}||\mathfrak{g}||_{1+a,T} \left\|\int_{0}^{t}\int_{\mathbb{R}^2}
	\left|(\nabla F)\left(s,Y_{s,t}(w,\frac{x-w}{t})+w+\frac{s(x-w)}{t}\right)\right|\frac{1}{\langle \frac{x-w}{t}\rangle^3}
	\frac{dwds}{\langle s\rangle^{a}t^2}\right\|_{L^p_x}.
\end{align*}
As \eqref{zz}, by Lemma  \ref{mathcal H1H2} one has 
\begin{align*}
	& 	||\delta_{\alpha}\mathcal{T}^{1,i}[F,\eta](t)||_{L^p}+	||\delta_{\alpha}\mathcal{T}^{2,i}[F,\eta](t)||_{L^p} \\&\lesssim_{a}t^{-1-\frac{2(p-1)}{p}}||\mathfrak{g}||_{1+a,T} \sum_{j=0,1}\left(\sup_{s\in [0,t]}\langle s\rangle^j\|\nabla^jF(s)\|_{L^1}+ \sup_{s\in [0,t]}\langle s\rangle^{j+\frac{2(p-1)}{p}}\|\nabla^jF(s)\|_{L^p}\right).
\end{align*}
This implies 
\begin{align}
	\sum_{p=1,\infty}\sum_{k=1,2}	\langle t\rangle^{1+a+\frac{2(p-1)}{p}}\frac{||\delta_{\alpha}\mathcal{T}^{k,i}[F,\eta](t)||_{L^p}}{|\alpha|^a}\lesssim_{a}||\mathfrak{g}||_{1+a,T} \sum_{j=0,1}\sum_{p=1,\infty}\sup_{s\in [0,t]}\langle s\rangle^{j+\frac{2(p-1)}{p}}\|\nabla^jF(s)\|_{L^p}\label{z10}.
\end{align}
\textit{Case 2:} $|\alpha|\leq  t$. As Step 2, thanks to \eqref{results of Y small}, \eqref{results of W small} and \eqref{z9}, one has 
\begin{align*}
&|\delta_{\alpha}	\mathcal{T}^{1,i}[F,\eta](t,x)|\\&\lesssim_{a} \frac{|\alpha|^a}{t^{1+a}}||\mathfrak{g}||_{1+a,T}\int_{0}^{t}\int_{\mathbb{R}^2}\sup_z|z|^{-a}\left|\delta_{z}
F(s,.)(Y_{s,t}(w,\frac{x-w}{t})+w+\frac{s(x-w)}{t})\right|\frac{dwds}{\langle s\rangle\langle\frac{x-w}{t} \rangle^3t^{2}}\\&\quad+\frac{|\alpha|^a}{t^{1+a}}||\mathfrak{g}||_{1+a,T}\int_{0}^{t}\int_{\mathbb{R}^2}\left|
F(s,Y_{s,t}(w,\frac{x-w}{t})+w+\frac{s(x-w)}{t})\right|\frac{dwds}{\langle s\rangle^{a}\langle\frac{x-w}{t} \rangle^3t^{2}},
\end{align*}
\begin{align*}
	&|\delta_{\alpha}	\mathcal{T}^{2,i}[F,\eta](t,x)|\\&\lesssim_{a} \frac{|\alpha|^a}{t^{1+a}}||\mathfrak{g}||_{1+a,T}\int_{0}^{t}\int_{\mathbb{R}^2}\sup_z|z|^{-a}\left|\delta_{z}
	(\nabla_xF)(s,.)(Y_{s,t}(w,\frac{x-w}{t})+w+\frac{s(x-w)}{t})\right|\frac{dwds}{\langle\frac{x-w}{t} \rangle^3t^{2}}\\&\quad+\frac{|\alpha|^a}{t^{1+a}}||\mathfrak{g}||_{1+a,T}\int_{0}^{t}\int_{\mathbb{R}^2}\left|
(\nabla_x	F)(s,Y_{s,t}(w,\frac{x-w}{t})+w+\frac{s(x-w)}{t})\right|\frac{dwds}{\langle\frac{x-w}{t} \rangle^3t^{2}}.
\end{align*}
By Lemma  \ref{mathcal H1H2},
\begin{align*}
&	||\delta_{\alpha}	\mathcal{T}^{1,i}[F,\eta](t)||_{L^p}\lesssim_{a} \frac{|\alpha|^a}{t^{\frac{2(p-1)}{p}+1+a}}||\mathfrak{g}||_{1+a,T}\left(\sup_{s\in [0,t]}\langle s\rangle\|F(s)\|_{\dot{F}^a_{1,\infty}}\right.\\&\left.\quad\quad\quad\quad+ \sup_{s\in [0,t]}\langle s\rangle^{1+\frac{2(p-1)}{p}}\|F(s)\|_{\dot{F}^a_{p,\infty}}+\sup_{s\in [0,t]}\langle s\rangle\|F(s)\|_{L^1}+ \sup_{s\in [0,t]}\langle s\rangle^{1+\frac{2(p-1)}{p}}\|F(s)\|_{L^p}\right),
\end{align*}
\begin{align*}
&	||\delta_{\alpha}	\mathcal{T}^{2,i}[F,\eta](t)||_{L^p}\lesssim_{a} \frac{|\alpha|^a}{t^{\frac{2(p-1)}{p}+1+a}}||\mathfrak{g}||_{1+a,T}\left(\sup_{s\in [0,t]}\langle s\rangle^{1+\frac{a}{2}}\|\nabla F(s)\|_{\dot{F}^a_{1,\infty}}\right.\\&\left.+ \sup_{s\in [0,t]}\langle s\rangle^{1+\frac{a}{2}+\frac{2(p-1)}{p}}\| \nabla F(s)\|_{\dot{F}^a_{p,\infty}}+\sup_{s\in [0,t]}\langle s\rangle^{1+\frac{a}{2}}\|\nabla F(s)\|_{L^1}+ \sup_{s\in [0,t]}\langle s\rangle^{1+\frac{a}{2}+\frac{2(p-1)}{p}}\|\nabla F(s)\|_{L^p}\right).
\end{align*}
Thus, 
\begin{align*}
	&\sum_{p=1,\infty}\sum_{k=1,2}	\langle t\rangle^{1+a+\frac{2(p-1)}{p}}\frac{||\delta_{\alpha}\mathcal{T}^{k,i}[F,\eta](t)||_{L^p}}{|\alpha|^a}\\&\quad\lesssim_{a}||\mathfrak{g}||_{1+a,T} \sum_{j=0,1}\sum_{p=1,\infty}\sup_{s\in [0,t]}\langle s\rangle^{1+\frac{2(p-1)}{p}}\left(\langle s\rangle^{\frac{a}{2}}\|\nabla^jF(s)\|_{\dot{F}^a_{p,\infty}}+\langle s\rangle^{\frac{a}{2}}\|\nabla F(s)\|_{L^p}+\| F(s)\|_{L^p}\right).
\end{align*}
Combining this with \eqref{zz1}, \eqref{zz2} and \eqref{z10} to get
\begin{align*}
&\sum_{p=1,\infty}	t^{1+a+\frac{2(p-1)}{p}}	||\nabla\mathcal{T}[F,\eta](t)||_{\dot{B}^a_{p,\infty}}\\&\quad\quad\lesssim_{a}||\mathfrak{g}||_{1+a,T}^a\sum_{j=0,1}\sum_{p=1,\infty}\sup_{s\in [0,t]}\langle s\rangle^{1+\frac{2(p-1)}{p}}\left(\langle s\rangle^{\frac{a}{2}}\|\nabla^jF(s)\|_{\dot{F}^a_{p,\infty}}+\| F(s)\|_{L^p}\right),
\end{align*}
provided that $1\leq t\leq T$ and $T\geq 1$. So, we proved \eqref{z12b}. The proof is complete.  
		\end{proof}\vspace{0.5cm}\\
	We used the following lemma in the proof of Lemma \ref{esT}. It is Lemma 8.1 in \cite{HNX1}.
		\begin{lemma}\label{mathcal H1H2}
		Assume $\mathcal{H}\in L^1(\mathbb{R}^2)\cap L^\infty(\mathbb{R}^2)$ and  $\varphi(x,v)$ satisfies $\|\nabla_{x,v} \varphi(x,v)\|_{L^\infty_{x,v}}\leq\frac{1}{2}$, then for $p=1,\infty$, $0\leq s\leq t$ and $t\geq0$ we have
		\begin{align}\label{mathcal H p}
			\left\|\int_{\mathbb{R}^2}\mathcal{H}\left(\varphi(x,v)+x-(t-s)v\right)\frac{dv}{\langle v\rangle^3}\right\|_{L^p_x}\lesssim\|\mathcal{H}\|_{L^p}.
		\end{align}
		Moreover, for $0\leq s\leq \frac{t}{2}$ and $t\geq1$, we also get
		\begin{align}\label{mathcal H 1}
			\left\|\int_{\mathbb{R}^2}\mathcal{H}\left(\varphi(x,v)+x-(t-s)v\right)\frac{dv}{\langle v\rangle^3}\right\|_{L^p_x}\lesssim\frac{1}{t^{\frac{2(p-1)}{p}}}\|\mathcal{H}\|_{L^1}.
		\end{align}
	\end{lemma}
\vspace{0.1cm}
	
		\section{ Proof of Proposition \ref{propoglobal} and  Proposition \ref{propolocal}}
			Frist, we prove \textit{Propositions \ref{propoglobal}}.  Assume that  the problem \eqref{eq2} has a unique solution $(f,U)$ in $[0,T]$ for some $T<\infty$ with 
		\begin{equation*}
			\|\rho\|_{1+a,T}+\|U\|_{1+a,T}\leq \varepsilon_1.
		\end{equation*}
		Since $\mathfrak{g}=\rho+A(U)$ and Assumption 3, 
$$
			||\mathfrak{g}||_{1+a,T}\leq 	||\rho||_{1+a,T}+	||A(U)||_{1+a,T}\leq \varepsilon_1+C\varepsilon_1^{2}.$$
		Assume that  $\varepsilon_1+C\varepsilon_1^{2}\leq \varepsilon_0$. One has 
$$||\mathfrak{g}||_{1+a,T}\leq \varepsilon_0.$$
		Thus, we can apply Proposition \ref{estimates about Y and W} with $(\mathfrak{g}, \mathbf{E})=(\rho+A(U),E)$ and  Lemma \ref{I a estimate} with $f_0$ to obtain that 
$$||\mathcal{I}(\mathfrak{g})||_{1+a,T}\lesssim_a |||f_0|||_{1+a}.$$
		Applying Lemma \ref{esT} with $(F,\eta)=(E_i,\partial_{v_i}\mu)$ and thanks to
		$\mathcal{R}(\mathfrak{g})=\sum_{i=1}^{2}\mathcal{T}[E_i,\partial_{v_i}\mu]$, one gets
		\begin{align*}
			||\mathcal{R}(\mathfrak{g})||_{1+a,T}&\lesssim_a \|\mathfrak{g}\|_{1+a,T}^{a}\sum_{j=0,1}\sum_{p=1,\infty}\sup_{s\in [0,T]}\langle s\rangle^{1+\frac{2(p-1)}{p}}\left(\langle s\rangle^{\frac{a}{2}}\|\nabla^jE(s)\|_{F^a_{p,\infty}}+\| E(s)\|_{L^p}\right)\\&\overset{\eqref{E estimate by rho}}{\lesssim_a} \|\mathfrak{g}\|_{1+a,T}^{1+a}\lesssim  \|	\rho\|_{1+a,T}^{1+a}+\|A(U)\|_{1+a,T}^{1+a}
			\lesssim_a	\|	\rho\|_{1+a,T}^{1+a}+\|U\|_{1+a,T}^{2(1+a)}.
		\end{align*}
		On the other hand, by \eqref{equation rho} and Proposition \ref{ProG},
$$
			||\rho||_{1+a,T}\lesssim_a||\mathcal{I}(\mathfrak{g})||_{1+a,T}+ 	||\mathcal{R}(\mathfrak{g})||_{1+a,T}+||A(U)||_{1+a,T}.$$
		Thus, we get 
		\begin{align}\label{z2}
			||\rho||_{1+a,T}\lesssim_a |||f_0|||_{1+a}+\|	\rho\|_{1+a,T}^{1+a}+\|U\|_{1+a,T}^{2(1+a)}.
		\end{align}
		Since $U=(-\Delta+1)^{-1}(\rho+A(U))$, 
		\begin{align*}
			\|U\|_{1+a,T}\lesssim_a\|	\rho\|_{1+a,T}+\|	U\|_{1+a,T}^2\lesssim_a  \|	\rho\|_{1+a,T}+\varepsilon_1\|	U\|_{1+a,T}.
		\end{align*}
		This implies 
		\begin{align*}
			\|U\|_{1+a,T}\lesssim _a\|	\rho\|_{1+a,T},
		\end{align*}
		provided  $\varepsilon_1\in (0,\varepsilon_0/4)$ small.
		Combining this with \eqref{z2} to obtain that 
		\begin{align*}
			||\rho||_{1+a,T}\lesssim_a |||f_0|||_{1+a}+\|	\rho\|_{1+a,T}^{1+a}\lesssim  |||f_0|||_{1+a}+\varepsilon_1^a\|	\rho\|_{1+a,T}.
		\end{align*}
		Thus, 
$$||\rho||_{1+a,T}+	\|U\|_{1+a,T}\lesssim_a  |||f_0|||_{1+a},$$
		provided  $\varepsilon_1\in (0,\varepsilon_0/4)$ small. So, we proved Proposition \ref{propoglobal}. 
		\begin{remark} \label{Re}It follows from Proposition \ref{estimates about Y and W} that $||f(T)|||_{1+a}<\infty.$
			However, it is easy to check that if there exists $(f_0^\kappa)_\kappa\in C^\infty_c(\mathbb{R}^2_x\times\mathbb{R}_v^2)$ satisfying 
			$\lim_{\kappa\to 0}||| f_0-f^\kappa_0|||_{1+a}=0$, then we can construct $(f_T^\kappa)_\kappa\in C^\infty_c(\mathbb{R}^2_x\times\mathbb{R}_v^2)$  such that $\lim_{\kappa\to 0}||| f(T)-f_T^\kappa|||_{1+a}=0$.
		\end{remark}
	Next, we prove \textit{Proposition \ref{propolocal}}. Set 
		\begin{equation}
			S_{\varepsilon,T_0}:=\{\vartheta\in (L^{1}\cap L^\infty([0,T])\times\mathbb{R}^2): ||\vartheta||_{a,T_0}= ||\vartheta\mathbf{1}_{[0,T_0]}(s)||_{a}\leq \varepsilon\},
		\end{equation}
	for $T_0\in [0,1]$. Let $\tilde{\varepsilon}_0$  be in Lemma \ref{lem0} and $\varepsilon_0>0$ be in Proposition \ref{estimates about Y and W}. \\
Thanks to \eqref{zzz15}, we can define a map 
\begin{equation*}
	\mathcal{J}(\rho):=G*_{(t,x)}(\mathcal{I}_{f_0}(\mathfrak{g})+\mathcal{R}(\mathfrak{g})+A(U))+\mathcal{I}_{f_0}(\mathfrak{g})+\mathcal{R}(\mathfrak{g}), \mathfrak{g}=\rho+A(U), U=\mathcal{N}(\rho),
\end{equation*}
for $\rho\in S_{\tilde{\varepsilon}_0,T_0}$. \\
We will prove that $\mathcal{J}(\rho)$ has a unique fixed point  in $S_{\tilde{\varepsilon}_1,T_0}$ for some $\tilde{\varepsilon}_1\leq \tilde{\varepsilon}_0$ and $T_0\in (0,1)$. \\By \eqref{zzz18} in Lemma \ref{lem0}, for any $\rho\in S_{\tilde{\varepsilon}_1,T_0}$,
\begin{equation}
	||\mathfrak{g}||_{a,T_0}\leq C_0' ||\rho||_{a,T_0}\leq C_0'\tilde{\varepsilon}_1\leq \varepsilon_0,
\end{equation}
 with $\tilde{\varepsilon}_1\in (0,\tilde{\varepsilon}_0)$ small. 
Let $(Y_{s,t}^\rho,W_{s,t}^\rho)$ be the characteristics in Proposition \ref{estimates about Y and W} with $\mathfrak{g}=\rho+A(\mathcal{N}(\rho))$.  Using the same argument as \cite[Proof of Proposition 4.1]{HNX1}, we can obtain for any $0\leq s\leq t\leq T_0$,
	\begin{align}\label{zzz1}
		&\sum_{i=0,1}	\|\nabla_x^i(Y_{s,t}^\rho,W_{s,t}^\rho)\|_{L^\infty_{x,v}}+\sup_{\alpha}	\frac{	\|\delta_{\alpha}^x	\nabla_x(Y_{s,t}^\rho,W_{s,t}^\rho)\|_{L^\infty_{x,v}}}{|\alpha|^a}\lesssim_a T_0 \|\rho\|_{a,T_0},\\&\nonumber
	 \sum_{i=0,1}	\|\nabla_{x}^i(Y_{s,t}^{\rho_1}-Y_{s,t}^{\rho_2},W_{s,t}^{\rho_1}-W_{s,t}^{\rho_2})\|_{L^\infty_{x,v}}\\&\quad\quad\quad\quad\quad\quad+\sup_{\alpha}	\frac{	\|\delta_{\alpha}^x	\nabla_x(Y_{s,t}^{\rho_1}-Y_{s,t}^{\rho_2},W_{s,t}^{\rho_1}-W_{s,t}^{\rho_2})\|_{L^\infty_{x,v}}}{|\alpha|^a}\lesssim_a T_0 \|\rho_1-\rho_2\|_{a,T_0}\label{zzz2},
	\end{align}
for any $\rho,\rho_1,\rho_2\in S_{\tilde{\varepsilon}_1,T_0}$. It is easy to obtain from $\eqref{zzz1}$ that  for any $\rho\in S_{\tilde{\varepsilon}_1,T_0},$
\begin{align}\nonumber
&	||\mathcal{R}(\rho+A(\mathcal{N}(\rho)))||_{1+a,T_0}\lesssim_a T_0,~~
	\\&
	||\mathcal{I}_{f_0}(\rho+A(\mathcal{N}(\rho)))||_{1+a,T_0}\lesssim_a  ||| f_0|||_{1+a}.\label{zzz3}
\end{align}
Hence, since $\mathcal{J}(\rho)-\mathcal{I}_{f_0}(\mathfrak{g})=G*_{(t,x)}(\mathcal{I}_{f_0}(\mathfrak{g})+\mathcal{R}(\mathfrak{g})+A(U))+\mathcal{R}(\mathfrak{g})$,
\begin{equation}\label{zzz4}
		||\mathcal{J}(\rho)-\mathcal{I}_{f_0}(\mathfrak{g})||_{1+a,T_0} \lesssim T_0+(1+ ||| f_0|||_{1+a})\int_0^{T_0}||G(t)||_{L^1}dt\lesssim T_0(1+ ||| f_0|||_{1+a}),
\end{equation}
where we  use \cite[Theorem 3.6, (3.18)]{HNX1} in the last inequality. Hence
\begin{align}\nonumber
&	||\mathcal{J}(\rho)-\rho(0)||_{1+a,T_0}\leq ||\mathcal{J}(\rho)-\mathcal{I}_{f_0}(\mathfrak{g})||_{1+a,T_0}+	||(\mathcal{I}_{f_0}-\mathcal{I}_{f_0^\kappa})(\mathfrak{g})||_{1+a,T_0}\\&\nonumber+||\mathcal{I}_{f_0^\kappa}(\mathfrak{g})-\mathcal{I}_{f_0^\kappa}(0)||_{1+a,T_0}
	+||\mathcal{I}_{f_0^\kappa}(0)-\int_{\mathbb{R}^2}f_0^\kappa(x,v)dv||_{1+a,T_0}+||\int_{\mathbb{R}^2}(f_0^\kappa-f_0)(x,v)dv||_{1+a,T_0}\\&\overset{\eqref{zzz3}}\lesssim T_0(C(f_0^\kappa)+ ||| f_0|||_{1+a})+|||f_0-f_0^\kappa|||_{1+a}.\label{zzz12}
\end{align}
In particular, 
\begin{align}\nonumber
&	\sup_{\rho\in S_{\tilde{\varepsilon}_0,T_0} }	||\mathcal{I}_{f_0}(\rho+A(\mathcal{N}(\rho)))||_{a,T_0}\lesssim	T_0(C(f_0^\kappa)+ ||| f_0|||_{1+a})+|||f_0-f_0^\kappa|||_{1+a}+||\rho(0)||_{1+a,T_0}\\&\quad\quad
\lesssim T_0(C(f_0^\kappa)+ ||| f_0|||_{1+a})+|||f_0-f_0^\kappa|||_{1+a}+	 \sum_{p=1,\infty}(	||\rho(0)||_{\dot B^{a}_{p,\infty}}+||\rho(0)||_{L^p}).\label{zzz13}
\end{align}
Moreover, as \cite[Proof of Theorem 2.2]{HNX1}, we get from \eqref{zzz1}, \eqref{zzz2} and Lemma  \ref{lem0} that 
\begin{equation}\label{zzz14}	||\mathcal{J}(\rho_1)-\mathcal{J}(\rho_2)||_{a,T_0}\lesssim T_0(1+	||| f_0|||_{1+a})	||\rho_1-\rho_2||_{a,T_0},
\end{equation}
for any $\rho_1,\rho_2\in S_{\tilde{\varepsilon}_1,T_0}$. 
Thanks to \eqref{zzz13}, \eqref{zzz14}, \eqref{f_0condition} and the fixed-point theorem, we obtain that $\mathcal{J}$ has a unique fixed  point $\rho$ in $S_{\tilde{\varepsilon}_1,T_0}$ satisfying \eqref{zzz8} for some $T_0\in (0,1)$, $\tilde{\varepsilon}_1\in (0,\tilde{\varepsilon}_0)$ small enough. Clearly, \eqref{zzz77} follows from \eqref{zzz12}.
 The proof of the Proposition \eqref{propolocal} is complete.

	\end{document}